\documentclass[12pt,reqno]{amsart}
\usepackage{graphicx} 
\usepackage{amssymb,amscd, amsthm, latexsym, mathrsfs, epsfig, appendix}
\usepackage[all]{xy}
\usepackage{todonotes}
\usepackage[shortlabels]{enumitem}
\usepackage{pgfplots}

\newcommand\C{\mathbb C}

\newcommand\R{\mathbb R}
\newcommand\Z{\mathbb Z}

\newcommand{\re}{\operatorname{Re}}
\newcommand{\im}{\operatorname{Im}}

\newcommand{\hth}{\hat{\theta}}
\newcommand{\ii}{\sqrt{-1}}

\newcommand{\del}{\partial}
\newcommand{\delbar}{\overline{\partial}}
\newcommand{\Vol}{\operatorname{Vol}}

\newcommand{\rp}{\right)}
\newcommand{\lp}{\left(}
\newcommand{\sint}{\sin \hth}
\newcommand{\cost}{\cos \hth}
\newcommand{\cott}{\cot \hth}
\newcommand{\TT}{\tau}

\makeatletter \@addtoreset{equation}{section} \makeatother

\newtheorem{thm}{Theorem}
\newtheorem{prop}[thm]{Proposition}
\newtheorem{lem}[thm]{Lemma}
\newtheorem{cor}[thm]{Corollary}

\theoremstyle{definition}

\newtheorem{rmk}[thm]{Remark}

\newcommand\narrowdots{\hbox to 1em{.\hss.\hss.}}

\title[$\textrm{d}$HYM connections in a variable background]{Examples of $\textrm{d}$HYM connections\\ in a variable background}
\date{5 November 2020}
\author{Enrico Schlitzer}
\email{eschlitz@sissa.it}
\author{Jacopo Stoppa}
\email{jstoppa@sissa.it}
\address{SISSA, via Bonomea 265, 34136 Trieste, Italy}
\address{Institute for Geometry and Physics, Via Beirut 2, 34151 Trieste, Italy}
\begin{document}
\begin{abstract} We study deformed Hermitian Yang-Mills (dHYM) connections on ruled surfaces explicitly, using the momentum construction. As a main application we provide many new examples of dHYM connections coupled to a variable background K\"ahler metric. These are solutions of the moment map partial differential equations given by the Hamiltonian action of the extended gauge group, coupling the dHYM equation to the scalar curvature of the background. The large radius limit of these coupled equations is the K\"ahler-Yang-Mills system of \'Alvarez-C\'onsul, Garcia-Fernandez and Garc\'ia-Prada, and in this limit our solutions converge smoothly to those constructed by Keller and T{\o}nnesen-Friedman. We also discuss other aspects of our examples including conical singularities, realisation as B-branes, the small radius limit and canonical representatives of complexified K\"ahler classes.   
\end{abstract}

\maketitle
\section{Background and main results} 
\subsection{dHYM connections} Let $L \to X$ denote a holomorphic line bundle over a compact $n$-dimensional K\"ahler manifold, with a fixed background K\"ahler form $\omega$. A hermitian metric $h$ on the fibres of $L$ determines the two notions of \emph{Lagrangian phase} and \emph{radius} of the line bundle. Namely, writing $F(h) = \ii F$, $F \in \mathcal{A}^{1,1}(X, \R)$ for the curvature of the Chern connection, one introduces the endomorphism of the tangent bundle given by $\omega^{-1}F$, with eigenvalues $\lambda_i$. Then we have
\begin{align*}
\frac{(\omega - \ii F)^n}{\omega^n} &= \prod^n_{i=1}(1-\ii\lambda_i) 
= r_{\omega}(F) e^{\ii\Theta_{\omega}(F)},
\end{align*}
where the Lagrangian phase and radius are defined (using the background metric $\omega$) respectively as
\begin{align*}
&\Theta_{\omega}(F) = -\sum^n_{i=1}\arctan(\lambda_i),\,
 r_{\omega}(F) = \prod^n_{i=1}(1+\lambda^2_i)^{1/2}.
\end{align*} 
The \emph{deformed Hermitian Yang-Mills (dHYM) equation} (introduced in \cite{leungYau,marino} and surveyed in \cite{collinsShi, collinsXie}) is the condition of having constant Lagrangian phase, 
\begin{equation}\label{dHYMintro}
\Theta_{\omega}(F) = \hth \mod 2\pi.
\end{equation} 
The work of Conan Leung, Yau and Zaslow \cite{leungYau} shows that, at least under suitable assumptions, the dHYM equation \eqref{dHYMintro} is mapped to the special Lagrangian equation under mirror symmetry. Therefore this equation has attracted considerable interest in mathematical physics and complex differential geometry (see e.g. the foundational works \cite{collinsJacobYau, collinsYau, jacob} and the recent contributions \cite{hanJinStab, hanJinChern, takaCollapse, takaTan}).

In the present paper we study dHYM connections in the very special case when $X$ is a complex ruled surface. While this is a classical test bed for equations in complex differential geometry, here we allow a rather general setup, as we now discuss.

\subsection{Variable background dHYM} Most importantly, we couple the dHYM equation \eqref{dHYMintro} to a variable background K\"ahler metric $\omega$, through the equations  
\begin{align}\label{coupled_dHYM_Intro}
\begin{cases}
\Theta_{\omega}(F) = \hth \mod 2\pi\,\\
s(\omega) - \alpha r_{\omega}(F) = \hat{s} - \alpha \hat{r}, 
\end{cases} 
\end{align}
where $s(\omega)$, $\hat{s}$, $\hat{r}$ denote the scalar curvature and its average, respectively the average radius, and $\alpha \in \R$ is an arbitrary coupling constant. The quantities $\hat{s}$, $\hat{r}$ and $\hth$ are fixed by cohomology, and in particular we have
\begin{equation*}
\hat{r} = \frac{1}{n!\Vol(M,\omega)}\left|\int_{X} (\omega- F(h))^n\right|
\end{equation*}
and 
\begin{equation*}
e^{-\ii\hth}= \frac{1}{n!\Vol(M,\omega)\hat{r}}\int_{X} (\omega- F(h))^n.
\end{equation*}
The equations \eqref{coupled_dHYM_Intro} are obtained by combining the moment map pictures for dHYM connections (due to Thomas and Collins-Yau, see \cite{collinsYau, richard}) and for constant scalar curvature K\"ahler (cscK) metrics (due to Donaldson and Fujiki, see \cite{donaldson, fujiki}) in a very natural way, through the action of the extended gauge group (a canonical extension of the group of unitary gauge transformations by Hamiltonian symplectomorphisms): this is explained in \cite{us}, building on the results of \cite{AGG}. The coupling constant $\alpha$ is a scale parameter for the relevant symplectic form on the space of integrable connections. Thus, only the case when $\alpha > 0$ corresponds to a genuine K\"ahler reduction (rather than just a symplectic reduction). 

Let $\Sigma$ be a compact Riemann surface of genus $h$, with K\"ahler metric $g_{\Sigma}$ of constant scalar curvature $2 s_{\Sigma}$, and let $\mathcal{L} \xrightarrow{p}\Sigma$ denote a holomorphic line bundle of degree $k \in \Z_{> 0}$, with $2\pi c_1  \left(\mathcal{L} \right) = \left[ \omega_{\Sigma} \right]$. 
Since $\text{Vol}\left( \Sigma \right)= 2\pi k$, by the Gauss-Bonnet formula we have
\begin{equation*}
s_{\Sigma} = \frac{1}{\text{Vol}\left( \Sigma \right)}\int_{\Sigma} s_{\Sigma} \omega_{\Sigma} = \frac{1}{\text{Vol}\left( \Sigma \right)}\int_{\Sigma} \rho_{\Sigma}=\frac{2 \left( 1-h\right)}{k},
\end{equation*}
where $\rho_{\Sigma}$ denotes the Ricci 2-form of $g_{\Sigma}$.

We will construct solutions of the coupled equations \eqref{coupled_dHYM_Intro} on ruled surfaces of Hirzebruch type, obtained by the projectivization 
\begin{equation*}
X = \mathbb{P} \left(\mathcal{L}\oplus \mathcal{O} \right) \to \Sigma, 
\end{equation*}
where $\mathcal{O}$ denotes the trivial holomorphic line bundle. (It is well known that such $X$ does not admit cscK metrics). Our solutions are obtained by extending the classical \emph{momentum construction} (also known as the Calabi ansatz, see \cite{hwang}) to the equations \eqref{coupled_dHYM_Intro}: see \eqref{ansatzMetric}, \eqref{ansatzCurvature} for our ansatz.

Let $E_0 = \mathbb{P}\left(0\oplus \mathcal{O} \right)$ and $E_{\infty} = \mathbb{P}\left(\mathcal{L}\oplus 0 \right)$ denote respectively the zero section and the infinity section of the $\mathbb{CP}^1$-bundle $X$ over $\Sigma$, with general fibre $C$. We introduce the real parameters $k_1, k_2$ and $k' > 0$, and consider the cohomology classes
\begin{align}\label{CohClasses}
\nonumber [\omega] &= 2\pi[ 2 E_0 + k' C],\\
[ F ] &= 2\pi[2 (k_1 - k_2)  E_0  + (2 k k_2 + k' (k_1  + k_2 )  C], 
\end{align}
where we slightly abuse the notation and denote the Poincar\'e duals of $E_0$ and $C$ respectively by $[E_0]$ and $[C]$.
Then $[\omega]$ is a K\"ahler class and $[F/(2\pi)]$ is integral, provided $k_1, k_2, k'$ are integers and $k'>0$, so it is possible to find a holomorphic line bundle $L \to X$ such that $-2\pi c_1(L) =\left[F\right] $. 

\begin{rmk}\label{symmetryRmk}
The equation \eqref{dHYMintro} is equivalent to 
\begin{equation*}
\im \lp e^{-\ii \hth} \lp \omega-\ii F \rp^n \rp=0
\end{equation*}
and the latter condition is preserved when $F \to -F$ and $\hth \to -\hth$, which should be interpreted geometrically as considering the dHYM equation on $L^{-1}$ instead of $L$. With our choice of parametrization, this implies that the set of parameters corresponding to solutions of the system \eqref{coupled_dHYM_Intro} is invariant under  $k_i \to -k_i$, for $i=1,2$.
When $k_2=0$, it follows from \eqref{CohClasses} that, for any choice of K\"ahler class, the unique solution of the dHYM equation \eqref{dHYMintro} is given by $F= k_1\omega$.
In this case,  the Lagrangian radius is also constant $r_{\omega}(k_1 \omega)= (1+k_1^2)^2$ and, since $X$ does not admit cscK metrics, the system \eqref{coupled_dHYM_Intro} has no solution.
In Section 2 it will be clear that also for $k_1=0$ the dHYM equation has a trivial solution; in this case, $\hth =0$ and we can solve also \eqref{coupled_dHYM_Intro}. 
In the following, we will focus on the less trivial choices of parameters, assuming
\begin{equation*}
k_1 < 0, \qquad k_2 \neq 0. 
\end{equation*}
\end{rmk}
It is also convenient to introduce the quantity
\begin{equation*}
x = \frac{k}{k+k'} \in (0, 1).
\end{equation*}
\begin{thm}\label{MainThmSmooth} Suppose the ``stability condition"
\begin{equation}\label{stabilityRuledSurfaces}
( 1 + ( k_1 + k_2 )^2 ) > x ( 1 + ( k_1 -k_2 )^2 )
\end{equation}
holds. Then, there exist a unique K\"ahler form $\omega$ and curvature form $F$, with cohomology classes given by \eqref{CohClasses}, such that they are obtained by the momentum construction (see \eqref{ansatzMetric}, \eqref{ansatzCurvature}) and solve the coupled equations \eqref{coupled_dHYM_Intro} on the ruled surface $X$, for the unique value of the coupling constant 
\begin{equation*} 
\alpha = \frac{\sqrt{4 k_1^2 + ( 1 - k_1^2 + k_2^2)^2}}{2 ( 1 + ( k_1 - k_2)^2) k_2^2}\left( -2 + s_{\Sigma}x \right).
\end{equation*}
If equality holds instead in \eqref{stabilityRuledSurfaces}, then there is a smooth solution on $X\setminus E_{\infty}$, with underlying metric $\omega \in C^{1, 1/2}(X) \cap C^{\infty}(X\setminus E_{\infty})$.
\end{thm}
Theorem \ref{MainThmSmooth} is proved in Section \ref{MainThmSmoothSec}.
\subsection{Solutions with conical singularities} The main limitation of Theorem \ref{MainThmSmooth} concerns the sign of the coupling constant: it is straightforward to check that in the situation of that result we always have $\alpha < 0$, since $s_{\Sigma} \leq 2$ and $x<1$. In order to gain more flexibility we allow the background metric $\omega$ to develop conical singularities along the divisors $E_0$, $E_{\infty}$. Fix  $0 < \beta_0 \leq 1$ and let
\begin{equation*}
\beta_{\infty} = \frac{-2 + \beta_{0}(1+x)}{-1+x} \geq 1.
\end{equation*}
\begin{thm}\label{MainThmConic} Suppose the ``stability condition" \eqref{stabilityRuledSurfaces} holds. Then, there exist a unique K\"ahler form $\omega$ and curvature form $F$, such that they are obtained by the momentum construction (see \eqref{ansatzMetric}, \eqref{ansatzCurvature}), $\omega$ has conical singularities with cone angles $2\pi \beta_0$ along $E_0$ and $2\pi \beta_{\infty}$ along $E_{\infty}$, the corresponding cohomology classes (in the sense of currents) are given by \eqref{CohClasses}, and they solve the coupled equations \eqref{coupled_dHYM_Intro}, for the unique value of the coupling constant 
\begin{equation*}
\alpha = \frac{\sqrt{4 k_1^2 + ( 1 - k_1^2 + k_2^2)^2}}{2 ( 1 + ( k_1 - k_2)^2) k_2^2}\frac{3+x + s_{\Sigma}  x^2 - 3(1+x)\beta_0}{x}.
\end{equation*}
\end{thm}
Theorem \ref{MainThmConic} is proved in Section \ref{MainThmConicSec}. Note that this gives a generalisation of Theorem \ref{MainThmSmooth}: when $\beta_0 = 1$ we recover precisely the smooth solutions provided by that result.
\begin{cor} For sufficiently small cone angle $2\pi \beta_0$ and sufficiently large $k' > 0$, the coupling constant $\alpha$ is positive.
\end{cor}

\subsection{Relation to twisted KE metrics} As usual, under a suitable cohomological condition, the equation in \eqref{coupled_dHYM_Intro} involving the scalar curvature may be reduced to a condition involving the Ricci curvature. In our case, this condition is given by 
\begin{equation*}
[\operatorname{Ric}(\omega)] + \frac{\alpha}{2\sin\hth} [F] = \frac{\hat{s} - \alpha\hat{r}}{4} [\omega].    
\end{equation*} 
Then, the equation 
\begin{equation*}
s(\omega) - \alpha r_{\omega}(F) = \hat{s} - \alpha \hat{r}
\end{equation*}
reduces to the twisted K\"ahler-Einstein equation
\begin{equation}\label{twistedKE}
\operatorname{Ric}(\omega) + \frac{\alpha}{2\sin\hth}  F  = \frac{\hat{s} - \alpha\hat{r}}{4} \omega.
\end{equation}
We provide an explicit criterion for when this reduction occurs for the class of examples provided by Theorem \ref{MainThmConic} (in which case $\operatorname{Ric}(\omega)$, $F$ and $\omega$ extend to closed currents on $X$).  
\begin{prop}\label{tKEProp} The condition $s(\omega) - \alpha r_{\omega}(F) = \hat{s} - \alpha \hat{r}$ reduces to the twisted K\"ahler-Einstein equation \eqref{twistedKE} iff we have
\begin{align}\label{tKEPropCond}
\nonumber &(1 + k_1^2+ k_2^2) (x-1) \left(s_{\Sigma} x^2-3 \beta_0  (x+1)+x+3\right)\\&= 2 k_1 k_2 \left(-3
   \beta_0 +s_{\Sigma} x^3-x^2 (\beta_0 +s_{\Sigma}-1)+3\right).
\end{align}
Morevoer, there are infinitely many admissible values of  $k_1, k_2$, $k'$ which satisfy this equality for some $\beta_0$ and for which the ``stability condition" \eqref{stabilityRuledSurfaces} holds (so that the corresponding coupled equations are solvable).
\end{prop}
This result is proved in Section 6. Writing the dHYM equation on the surface $X$ in Monge-Amp\`ere form (as in \cite{collinsJacobYau}) we see that in the twisted K\"ahler-Einstein case the coupled equations \eqref{coupled_dHYM_Intro} become
\begin{align*} 
\begin{cases}
\left(- \sin(\hth ) F + \cos(\hth ) \omega\right)^2 = \omega^2\,\\
\operatorname{Ric}(\omega) + \frac{\alpha}{2\sin\hth}  F  = \frac{\hat{s} - \alpha\hat{r}}{4} \omega, 
\end{cases} 
\end{align*}
and so they are closely related to the systems of coupled Monge-Amp\`ere equations studied by Hultgren and Wytt-Nystr\"om \cite{wytt}.
\subsection{Realisation as B-branes} Given the origin of the dHYM equation in mirror symmetry, it seems interesting to ask whether the special dHYM connections appearing in Theorem \ref{MainThmSmooth}, i.e. solutions of the coupled equations \eqref{coupled_dHYM_Intro}, can in fact be realised as B-branes (i.e. for our purposes, holomorphic submanifolds endowed with a dHYM connection) in some ambient Calabi-Yau manifold (this is how the dHYM equation appears in mathematical physics, see e.g. \cite{collinsXie}). Thus we are asking for a Calabi-Yau manifold $\check{M}$ with a Ricci flat K\"ahler metric $\omega_{\check{M}}$, and a holomorphic embedding $\iota\!:X \hookrightarrow\check{M}$, such that the K\"ahler form $\omega$ constructed in Theorem \ref{coupled_dHYM_Intro} is given by the restriction $\omega =  \iota^*\omega_{\check{M}}$. We show that this can be achieved at least locally around $X$, relying on the classical results on Feix \cite{feix} on the hyperk\"ahler extension of real analytic K\"ahler metrics.
\begin{prop}\label{HKextensionProp} The K\"ahler form $\omega$ and curvature form $F$ provided by Theorem \ref{MainThmSmooth} are real analytic. Thus, $\omega$ extends to a hyperk\"ahler metric defined on an open neighbourhood of the zero section in the holomorphic cotangent bundle $T^*X$, and $F$ extends to the curvature form of a hyperholomorphic line bundle defined on the same open neighbourhood.  
\end{prop}  
This result is proved in Section \ref{MainThmSmoothSec}.
\subsection{Large and small radius limits} In the mathematical physics literature (see e.g. \cite{branes}, Chapter 1), the dHYM equation involves a ``slope" parameter $\alpha' > 0$ (related to the ``string length" by $\alpha' = l^2_s$), which appears simply as a scale parameter for the curvature form, $F \mapsto \alpha' F$. The corresponding coupled equations \eqref{coupled_dHYM_Intro} are given by
\begin{align}\label{coupled_dHYM_Intro_scale}
\begin{cases}
\Theta_{\omega}(\alpha' F) = \hth \mod 2\pi\,\\
s(\omega) - \alpha r_{\omega}(\alpha' F) = \hat{s} - \alpha \hat{r}. 
\end{cases} 
\end{align}
The expressions ``large radius limit" (or ``zero slope limit") refer to the behaviour of the dHYM equations and their solutions as $\alpha' \to 0$.  As explained in \cite{us}, the large radius limit of our coupled equations is the (rank $1$ case of) the K\"ahler-Yang-Mills system introduced by \'Alvarez-C\'onsul, Garcia-Fernandez and Garc\'ia-Prada \cite{AGG}. We can prove a much stronger result, at the level of solutions, on the ruled surface $X$. 
\begin{thm}\label{LargeRadiusThm} For all sufficiently small $\alpha'$, depending only on the fixed parameters $k_1, k_2$, $k'$, (i.e. on the fixed cohomology classes $[\omega]$, $[F]$), the coupled equations \eqref{coupled_dHYM_Intro_scale} are uniquely solvable on $X$ with the momentum construction. Moreover, as $\alpha'\to 0$, the corresponding solutions $\omega_{\alpha'}$, $F_{\alpha'}$ converge smoothly to a solution of the K\"ahler-Yang-Mills system
\begin{align}\label{KYM}
\begin{cases}
\Lambda_{\omega} F = \mu\\
s(\omega) + \tilde{\alpha}\,\Lambda^2_{\omega}(F\wedge F) = c
\end{cases} 
\end{align}   
for some (explicit) coupling constant $\tilde{\alpha}$.
\end{thm}
The particular solutions of the K\"ahler-Yang-Mills system obtained in this limit are due to Keller and T{\o}nnesen-Friedman \cite{keller}. 

Similarly, the ``small radius limit" (or ``infinite slope limit") concerns the behaviour of the coupled equations \eqref{coupled_dHYM_Intro_scale} as $\alpha' \to \infty$. 
\begin{thm}\label{SmallRadiusThm} Fix parameters $k_1, k_2$, $k'$ (i.e. cohomology classes $[\omega]$, $[F]$) such that the ``stability condition"
\begin{equation*}
 ( k_1 + k_2 )^2 > x  ( k_1 -k_2 )^2 
\end{equation*}
holds. Then the coupled equations \eqref{coupled_dHYM_Intro_scale} are uniquely solvable on $X$ with the momentum construction, for all $\alpha'>0$. Moreover, as $\alpha'\to \infty$, the corresponding solutions $\omega_{\alpha'}$, $F_{\alpha'}$ converge smoothly to a solution of the system
\begin{align*} 
\begin{cases}
F \wedge \omega = c_1 F^2\\
s(\omega) - \hat{\alpha}\,\Lambda_{\omega} F = c_2. 
\end{cases} 
\end{align*}
for some (explicit) coupling constant $\hat{\alpha}$.
\end{thm}
At least in the case when $F$ is K\"ahler, this system couples the $J$-equation $\Lambda_{F} \omega = c'_1$ for $F$ to a twisted cscK equation for $\omega$. In general, these limiting equations belong to a class of coupled PDEs studied by Datar and Pingali \cite{datar}.

Theorems \ref{LargeRadiusThm} and \ref{SmallRadiusThm} are proved in Section \ref{LargeRadiusSec}.
\subsection{Complexified K\"ahler classes} Complexified K\"ahler classes are expressions of the form 
$[\omega + \ii B]$, where $\omega$ is a K\"ahler form and $[B] \in H^2(M, \R)/H^2(M,\Z)$ is known as the B-field. They play an important role in mirror symmetry (see e.g. \cite{richardSurvey} Section 2). Let $M$ be a compact K\"ahler manifold with no holomorphic $2$-forms. Collins and Yau (\cite{collinsYau} Section 8) consider a dHYM equation on $M$ of the form
\begin{equation*}
\Theta_{\omega}(F + B) = \hth \mod 2\pi,
\end{equation*} 
where $\ii F$ is the unknown curvature form of a Hermitian holomorphic line bundle $L \to M$ and $B$ is a fixed representative of a (lift of a) B-field. Arguing from mirror symmetry, they propose that the existence of a solution $F$ should be related, conjecturally, to the a suitable notion of stability of the object $L$ with respect to the complexified K\"ahler class $[\omega + \ii B]$.  

In the special case then $L$ is the trivial bundle $\mathcal{O}_M$, the equation becomes 
\begin{equation*}
\Theta_{\omega}(B + \ii \del\delbar u) = \hth \mod 2\pi,
\end{equation*} 
so we are effectively trying to find a canonical representative of the B-field $[B]$ with respect to a background K\"ahler form $\omega$; the existence of such a representative should be related to the stability of the object $\mathcal{O}_M$ with respect to $[\omega + \ii B]$.

Our coupled equations 
\begin{align}\label{coupled_dHYM_Intro_Bfield}
\begin{cases}
\Theta_{\omega}(B) = \hth \mod 2\pi\,\\
s(\omega) - \alpha r_{\omega}(B) = \hat{s} - \alpha \hat{r}, 
\end{cases} 
\end{align}
with $[B]$ a (lift of a) class in $H^2(M, \R)/H^2(M,\Z)$, can then be thought of as trying to prescribe a canonical representative of the complexified K\"ahler class $[\omega + \ii B]$. Note that in the Calabi-Yau case, at zero coupling $\alpha = 0$ and in the large radius limit, these equations for the complex form $\omega + \ii B$ reduce to the conditions 
\begin{align*} 
\begin{cases}
\Delta_{\omega} B = 0\\
\operatorname{Ric}(\omega) = 0, 
\end{cases} 
\end{align*}
which are standard in the physics literature (see e.g. \cite{branes} Section 1.1).\\

As an example we shall discuss the existence of such a canonical representative for the complexified K\"ahler class
\begin{equation*}
[\omega  + \ii B] = 2\pi (2 E_0 + (k' + \ii k'') C) 
\end{equation*}
on our ruled surfaces $X$, where the K\"ahler condition is equivalent to $k' > 0$. The key observation is that this can be expressed in the form
\begin{align*} 
[\omega] &= 2\pi[ 2 E_0 + k' C],\\
[ B ] &= 2\pi[2 (k_1 - k_2)  E_0  + (2 k k_2 + k' (k_1  + k_2 )  C] 
\end{align*} 
with the special choices 
\begin{equation*}
k_1 = k_2 = \frac{k''}{2(k + k')},
\end{equation*}
provided we have $k'' < 0$. Thus we may apply Theorem \ref{MainThmSmooth} (and, more generally, Theorem \ref{MainThmConic} in the case of conical singularities) to show that the coupled equations \eqref{coupled_dHYM_Intro_Bfield} are solvable, uniquely under the momentum construction, iff the ``stability condition"
\begin{equation*}
( 1 + ( k_1 + k_2 )^2 )  = 1 + \left(\frac{k''}{k + k'}\right)^2 > x ( 1 + ( k_1 -k_2 )^2 ) = \frac{k}{k+k'} 
\end{equation*}
holds. But, clearly, this is automatically satisfied. By Remark \ref{symmetryRmk}, the same argument works for the case $k'' > 0$. 
\begin{cor} The complexified K\"ahler class 
\begin{equation*}
[\omega + \ii B] = 2\pi (2 E_0 + (k' + \ii k'') C),
\end{equation*}
where $k' > 0$, $k'' \neq 0 $, admits a canonical representative. This also holds allowing conical singularities; the corresponding coupling constant is given by
\begin{equation*}
\alpha = \frac{2 \sqrt{(k+k')^2+(k'')^2} \left(k^2 (-6 \beta_0 +s_{\Sigma}+4)+(7-9 \beta_0 ) k
   k'-3 (\beta_0 -1) (k')^2\right)}{k (k'')^2}.
\end{equation*}
\end{cor}
Note that a canonical representative with vanishing B-field $B = 0$ would correspond to a cscK metric, which does not exist. The coupling constant $\alpha$ diverges as $k'' \to 0$. It seems interesting that a nontrivial B-field can stabilise the unstable ruled surface $X$.\\

\noindent \textbf{Plan of the paper.} In Section \ref{momentumSec} we set up the momentum construction on our ruled surfaces. Section \ref{dHYMSec} solves the dHYM equation on our ruled surfaces explicitly using the momentum construction, under the necessary ``stability condition" \eqref{stabilityRuledSurfaces}. This result is applied in Section \ref{MainThmSmoothSec} in order to solve the coupled equations \eqref{coupled_dHYM_Intro}. All of this is extended to allow conical singularities in Section \ref{MainThmConicSec}; the main advantage is that in this case there exist solutions with positive coupling constants. Finally Section \ref{LargeRadiusSec} contains our results on the large and small radius limits.\\

\noindent \textbf{Acknowledgements.} We are grateful to Claudio Arezzo, Zak Dyrefelt, Annamaria Ortu and Carlo Scarpa for some discussions related to the material presented here. 

\section{Momentum construction}\label{momentumSec}
Let $X = \mathbb{P} \left(\mathcal{L}\oplus \mathcal{O} \right) \to \Sigma$ be a ruled surface as in the Introduction. Let 
\begin{equation*}
E_0 = \mathbb{P}\left(0\oplus \mathcal{O} \right),\,E_{\infty} = \mathbb{P}\left(\mathcal{L}\oplus 0 \right)
\end{equation*}
denote respectively the zero section and the infinity section of the $\mathbb{CP}^1$-bundle $X$ over $\Sigma$, with general fibre $C$. We have the straightforward intersection formulae:
\begin{equation}\label{intersection}
E_0 \cdot E_0 =-E_{\infty} \cdot E_{\infty}=k,\quad C \cdot C =0, \quad C \cdot  E_0 = C \cdot E_{\infty} =1.
\end{equation}
We will follow the standard \emph{momentum construction} (sometimes called the Calabi ansatz, see e.g. \cite{hwang}) for metrics on the complement of the zero section $X_0= \mathcal{L} \setminus E_0$, which extend across the zero and infinity sections of $X$ under suitable conditions.

Thus we consider metrics of the form
\begin{equation}\label{ansatzMetric}
\omega = \frac{p^*\omega_{\Sigma}}{x}+\ii \del \delbar {f}(s),
\end{equation}
where $x$ is a real parameter satisfying $0<x<1$, while ${f}$ is a strictly convex function, such that ${f}' : X_0 \to \left(-1,1 \right)$.
The real coordinate $s$ is the log-norm of the Hermitian metric $h(z)$ on $\mathcal{L}$ for which $-\del_z \delbar_z \log(h) = F(h) = -\ii\omega_{\Sigma}$. Considering a trivialization $U \subset\mathcal{L}$ with adapted bundle coordinates $\left(z,w \right)$, $s$ is given by
\begin{equation*}
s = \text{log} |\left(z,w \right)|^2_h = \text{log } |w|^2 + \text{log } h(z),
\end{equation*}
and it follows that
\begin{equation*}
\ii \del_{w} \delbar_{w}f(s)= \ii {f''(s)}\frac{dw \wedge d\bar{w}}{|w|^2}
\end{equation*}
and
\begin{equation*}
\ii \del_{z} \delbar_{z}f(s)= -f'(s)\omega_{\Sigma} + \ii f''(s)\frac{\del_z h \delbar_z h}{h^2}
\end{equation*}

If we choose $U$ such that $d \text{log } h (z_0)=0$ in $\left(z_0,w_0 \right)$, at this point all the mixed derivatives vanish and so we find
\begin{equation*}
\omega = \frac{1-x {f}'(s)}{x}\omega_{\Sigma} + \ii {f}''(s)\frac{dw \wedge d\bar{w}}{|w|^2};
\end{equation*}
moreover we also have, globally,
\begin{equation*}
\omega^2 = \frac{2}{|w|^2}\frac{1-x {f}'(s)}{x}{f}''(s)\omega_{\Sigma}\wedge\ii{dw \wedge d\bar{w}}.
\end{equation*}
Since $f(s)$ is strictly convex, we may consider its Legendre transform $u(\tau)$, a function of the
variable $\TT = f'(s)$, and define the \textit{momentum profile} 
\begin{equation*}
\phi (\TT) = \frac{1}{u''(\tau)}= f''(s),
\end{equation*}
which must satisfy the condition 
\begin{equation}\label{positivityConditionT}
\phi(\TT) > 0, \quad -1< \TT <1,
\end{equation}
required for $\omega$ to be positive.
Moreover the momentum construction shows that in order to extend $\omega$ across $w=0$ and $w= \infty$, 
$\phi (\TT)$ must satisfy the boundary conditions
\begin{equation}\label{boundaryConditionsT}
\lim_{\TT\to \pm 1} \phi (\TT) = 0, \quad \lim_{\TT\to \pm 1} \phi' (\TT) = \mp 1.
\end{equation}
The space $H^2(X, \R)$ is generated by the Poincar\'e duals of $E_0$ and $C$. Following \cite{keller}, we define the $2$-form 
\begin{equation*}
\beta = \frac{x^2}{\left( 1-xf'(s)\right)^2}\left(\frac{1-x {f}'(s)}{x}\omega_{\Sigma} - \ii {f}''(s)\frac{dw \wedge d\bar{w}}{|w|^2} \right).
\end{equation*}
A direct computation shows that $\beta$ is a closed $(1,1)$-form, traceless with respect to $\omega$, and  $\left\{ \omega, \beta \right\}$ is a basis for the space $H^2(X, \R)$.
We consider now a real $(1,1)$ cohomology class and its representative 
\begin{equation}\label{curvatureBasis}
F_0 = c_1 \omega + c_2 \beta.
\end{equation}
In order to identify $\ii F_0$ with the curvature form of a connection on some line bundle over $X$, $ [ F_0/(2\pi) ]$ must be an integral class. For $\left[ F_0 \right] = a \left[E_0 \right]  + b \left[C \right] $, using the identities \eqref{intersection}, we have
\begin{equation} \label{explicitCoordinates}
a = \int_CF , \qquad b = \int_{E_0} F_0 - k \int_C F_0.
\end{equation}
Since $E_0 = (f')^{-1}(-1)$, we get
\begin{equation*}
\int_{E_0} \omega = \frac{(1+x)}{x}\int_{\Sigma}\omega_{\Sigma} = 2\pi k \frac{(1+x)}{x}
\end{equation*}
and 
\begin{equation*}
\int_{E_0} \beta = \frac{x}{(1+x)}\int_{\Sigma}\omega_{\Sigma} = 2\pi k \frac{x}{(1+x)}.
\end{equation*}
For the general fibre $C$, let $w$ denote the bundle adapted coordinate along the fibre and define $r = |w|$, such that $s = 2\log r $ and $d/ds = \frac{r}{2}d/dr$. Using the boundary conditions \eqref{boundaryConditionsT}, we have
\begin{align*}
\int_C \omega &= \int_{\C \setminus \left\{0 \right\} }\ii {f}''(s)\frac{dw \wedge d\bar{w}}{|w|^2}
\\&=\int_{-\infty}^{+\infty}\int_0^{2\pi} \frac{d}{dr}f'(s)dr\wedge d\theta
\\&=2\pi \left(\lim_{s\to \infty}f'(s) - \lim_{s\to -\infty}f'(s) \right)
\\&= 4\pi
\end{align*}
and similarly
\begin{align*}
\int_C \beta = -4\pi \frac{x^2}{1-x^2}.
\end{align*}

Using \eqref{explicitCoordinates}, we obtain
\begin{equation*}
\left[ \frac{F_0}{2\pi} \right] = \left( 2c_1-2\frac{x^2}{1-x^2}c_2\right) E_0 + \left( \frac{1-x}{x}kc_1 + \frac{x}{1-x}kc_2 \right)C.
\end{equation*}
If we introduce the new parametrisation  
\begin{equation}\label{integralityConditions}
x=\frac{k}{k+k'}, \; c_1 = k_1, \; c_2 = \frac{1-x^2}{x^2}k_2,
\end{equation}
for real $k_1$, $k_2$ and $k' > 0$, then a direct calculation shows that the cohomology classes of $[\omega]$ and $[F_0]$ are given by our previous formulae
\begin{align*} 
[\omega] &= 2\pi[ 2 E_0 + k' C],\\
[ F_0 ] &= 2\pi[2 (k_1 - k_2)  E_0  + (2 k k_2 + k' (k_1  + k_2 )  C].
\end{align*} 
In particular we see that the choices $k' \in \Z_{>0}$ and $k_i \in \Z$ for $i=1,2$ correspond to integral classes.

\section{$\rm{d}$HYM on ruled surfaces}\label{dHYMSec}

In this Section we will solve the dHYM equation \eqref{dHYMintro} on $X$ explicitly, with respect to a fixed K\"ahler metric $\omega$ obtained by the momentum construction \eqref{ansatzMetric}. Given a class $\left[F\right]$ satisfying the integrality conditions \eqref{integralityConditions}, we may fix a holomorphic line bundle $L\to X$  with first Chern class $-2\pi\left[ c_1(L)\right]=\left[F\right] $.

Recall that the parameter $\hth$ is a topological constant determined by the condition 
\begin{equation*}
\int_X (\omega-\ii F)^2 \in \R_{>0}e^{\ii \hth}.
\end{equation*}
\begin{lem} We have
\begin{equation*}
e^{\ii \hth}= \frac{\left(1-k_1^2+k_2^2 -2\ii  k_1 \right)}{\sqrt{\left(1-k_1^2+k_2^2\right)^2 +\left(2 k_1 \right)^2}}.
\end{equation*}
\end{lem}
\begin{proof} Since $\beta$ is traceless with respect to $\omega$, we only need to compute the quantities $\int_X \omega^2$, $\int_X \beta^2$. We have 
\begin{align*}
\int_X \omega^2 &= 2\int_X f''(s) \frac{( 1- xf'(s) )}{x}\omega_{\Sigma}\wedge \frac{dw \wedge d\bar{w}}{|w|^2}
\\&=4\pi \int_{\Sigma}\omega_{\Sigma}\int_{0}^{\infty}\frac{d}{dr}( \frac{f'(s)}{x}+\frac{( f'(s))^2}{2} ) dr
\\&=\frac{16\pi^2 k}{x}
\end{align*}
and similarly
\begin{equation*}
\int_X \beta^2 = -\frac{16\pi^2 k}{x}\frac{x^4}{(1-x^2)^2}
\end{equation*}
Using \eqref{integralityConditions}, we find
\begin{equation*}
\int_X \left(\omega-\ii F\right)^2 = \frac{16\pi^2 k}{x}\left(1-k_1^2+k_2^2 -2\ii  k_1 \right),
\end{equation*} 
from which the claim follows immediately.
\end{proof}
In order to solve the dHYM equation in the class $\left[F_0\right]$ we extend the momentum construction by making the ansatz
\begin{equation}\label{ansatzCurvature}
F = F_g = F_0 + \ii \del \delbar g(s).
\end{equation}
It will be convenient to introduce the function $\nu (\tau)$ given by the image of $g'(s)$ under the Legendre transform diffeomorphism relative to $f(s)$.
\begin{lem} The form $\ii \del \delbar g(s)$ extends smoothly to an exact form on $X$ iff $\nu(\tau)$ extends smoothly to the interval $[-1,1]$ and vanishes at the boundary points.  
\end{lem}
\begin{proof} The component of $\ii \del \delbar g(s)$ in the fibre direction is
\begin{equation*}
\ii g''(s)\frac{dw \wedge d\bar{w}}{|w|^2}= \ii \nu'(\TT)\phi(\TT)\frac{dw \wedge d\bar{w}}{|w|^2}.
\end{equation*}
So $\ii \del \delbar g(s)$ extends smoothly to $X$ iff $\nu(\tau)$ extends smoothly to $[-1,1]$. In order to derive the appropriate boundary behaviour so that this extension is still exact, we compute 
\begin{equation*}
\int_{E_0} \del \delbar g = -2\pi k\left( \lim_{s\to -\infty}g'(s)\right)
\end{equation*}
and
\begin{equation*}
\int_C \del \delbar g = 2\pi \left(\lim_{s\to \infty}g'(s) - \lim_{s\to -\infty}g'(s) \right).
\end{equation*}
Using \eqref{explicitCoordinates}, the only conditions we need to impose are
\begin{equation}\label{exactnessConditions}
\lim_ {\TT \to \pm 1} \nu(\TT) = 0.
\end{equation}
\end{proof}
Our next result shows how to reduce the dHYM equation to an ODE. It is convenient to introduce the new variable 
\begin{equation*}
t = 1/x-\TT
\end{equation*}
as well as the auxiliary function
\begin{equation}\label{definitionH}
H(t) =  k_1 t + \frac{k_2}{t}\frac{1-x^2}{x^2} -\nu(t).
\end{equation}
\begin{prop}\label{dHYMODEProp} Under the momentum construction \eqref{ansatzMetric}, \eqref{ansatzCurvature}, the dHYM equation is equivalent to the ODE
\begin{equation}\label{HdHYM}
H'(t) = \frac{t\sint  +H(t)\cost }{H(t)\sint -t\cost },
\end{equation}
together with the boundary conditions
\begin{align}\label{bdryH}
\nonumber &H\lp\frac{1+x}{x} \rp = k_1\lp\frac{1+x}{x} \rp+k_2\lp\frac{1-x}{x} \rp, \\
&H\lp\frac{1-x}{x} \rp = k_1\lp\frac{1-x}{x} \rp+k_2\lp\frac{1+x}{x} \rp.
\end{align}
\end{prop}
\begin{proof} At a point $(z_0,w_0)$ such that $d \text{log } h (z_0)=0$, we have
\begin{align*}
&\omega -i F_g = \\&\left( \left(1-\ii k_1 \right)\frac{1-xf'}{x} -\ii \frac{k_2}{x} \frac{ 1-x^2}{ 1-xf'}+\ii g' \right)\omega_{\Sigma}+
\\& \left( f''\left(1-\ii k_1 +\ii k_2 \frac{ 1-x^2}{ 1-xf'}  \right) -\ii g''\right)\ii\frac{dw \wedge d\bar{w}}{|w|^2}.
\end{align*}
and we obtain the global identity
\begin{align}\label{dHYMsVariable}
\nonumber&\frac{1}{2}\im \lp e^{-\ii \hth}\lp \omega-\ii F_g\rp^2 \rp / \ii\frac{dw \wedge d\bar{w}}{|w|^2}\wedge\omega_{\Sigma} =  \\
\nonumber&-\sint \lp f'' \frac{1-xf'}{x}+\lp g'-\frac{k_2}{x}\frac{ 1-x^2}{ 1-xf'} -\frac{k_1}{x}+k_1f'\rp \lp g''+k_1f''-k_2f''\frac{ 1-x^2}{ \lp 1-xf' \rp^2} \rp \rp \\
&+\cost \lp \frac{1-xf'}{x} \lp k_2f''\frac{ 1-x^2}{ \lp 1-xf' \rp^2} -g''-k_1f''\rp +f''\lp  g'-\frac{k_2}{x}\frac{ 1-x^2}{ 1-xf'}-\frac{k_1}{x} +k_1f'\rp \rp .
\end{align}
This expression becomes much simpler under the Legendre transform diffeomorphism in terms of the variable $\TT = f'(s)$, for which ${d\TT}/{ds} = \phi(\TT)$, and the additional affine change of variable $t = 1/x-\TT$. Setting 
\begin{equation*}
H(t) =  k_1 t+\frac{k_2}{t}\frac{1-x^2}{x^2} -\nu(t),
\end{equation*}
the dHYM equation is equivalent to
\begin{equation*}
2\phi\lp \cost \lp  H+tH'\rp+\sint \lp t-HH'\rp \rp=0
\end{equation*}
and, since $\phi >0$, also to
\begin{equation*} 
H' = \frac{t\sint  +H\cost }{H\sint -t\cost }.
\end{equation*}
A direct computation shows that the boundary conditions \eqref{exactnessConditions} for $g(s)$, rephrased in term of $H(s)$, become the constraints \eqref{bdryH}.
\end{proof}
\begin{cor} The ODE \eqref{HdHYM} is solvable with the boundary conditions \eqref{bdryH} iff the ``stability condition"
\begin{equation*} 
\lp 1 + \lp k_1 + k_2 \rp^2 \rp > x \lp 1 + \lp k_1 -k_2 \rp^2 \rp
\end{equation*}
holds.
\end{cor}
\begin{proof} Setting $tv=H$, equation \eqref{HdHYM} becomes
\begin{equation}\label{vdHYM}
tv' = -2 \frac{\xi(v)}{\xi'(v)},
\end{equation}
with $\xi(v)= v^2\sint -2v\cost-\sint$.
Solving \eqref{vdHYM} by separation of variables, we get
\begin{equation*}
\xi(v) = \frac{C}{t^2},
\end{equation*}
which has two solutions given by
\begin{equation}\label{signSolutions}
H_{\pm}(t) = t \cott \pm \sqrt{\lp\cot^2\hth+1\rp\lp t^2+C'\rp},
\end{equation}
with $C' = C\sint$.
We need to impose the appropriate boundary conditions \eqref{bdryH}. The first condition at ${1}/{x}+1$ holds iff   we choose the solution $H_-$ in \eqref{signSolutions} and set
\begin{equation*}
C=\frac{-2k_2 \lp 1 + \lp k_1 + k_2\rp^2 - x^2 - \lp k_1 - k_2\rp^2 x^2\rp}{x^2\sqrt{\left(1-k_1^2+k_2^2\right)^2 +\left(2 k_1 \right)^2}}.
\end{equation*}
In this case, at ${1}/{x}-1$ we have
\begin{align*}
& H_- \lp \frac{1-x}{x} \rp = \frac{1}{2xk_1}\lp -k_1^2\lp -1+x \rp + \lp 1 +k_2^2\rp \lp -1 +x \rp\rp \\&+  \left| \frac{-1-\lp k_1 + k_2 \rp^2 +x+ \lp k_1-k_2 \rp^2x}{2k_1x}\right|\\
&= \begin{cases} k_1\lp\frac{1-x}{x} \rp+k_2\lp\frac{1+x}{x} \rp &\mbox{if } \lp 1 + \lp k_1 + k_2 \rp^2 \rp > x \lp 1 + \lp k_1 -k_2 \rp^2 \rp \\
\frac{\lp 1+ k_2^2 \rp \lp x -1 \rp-k_1 k_2 \lp 1 + x \rp}{k_1 x} & \mbox{if } \lp 1 + \lp k_1 + k_2 \rp^2 \rp < x \lp 1 + \lp k_1 -k_2 \rp^2 \rp, \end{cases}
\end{align*}
so the second condition in \eqref{bdryH} holds iff we have
\begin{equation*} 
\lp 1 + \lp k_1 + k_2 \rp^2 \rp > x \lp 1 + \lp k_1 -k_2 \rp^2 \rp.
\end{equation*}
\end{proof}
\begin{rmk} Jacob and Yau \cite{jacob} showed that the solvability of the dHYM equation on compact K\"ahler surfaces is equivalent to a certain numerical ``stability condition". Considering the closed, real $(1,1)$-form
\begin{equation*}
\Omega = \cot\hth\,\omega - F,
\end{equation*}
the relevant condition is  $\left[ \Omega \right]>0$.
In our setting, when we regard $H^2(X, \mathbb{R})$ as $\mathbb{R}^2$ with the basis provided by the Poincar\'e duals of $E_0$ and $C$ and coordinates $(a_1,a_2)$, the K\"ahler cone is identified with the subset $\{a_1 >0, a_2>0\}$. A computation shows that the $\left[ \Omega \right]$ is positive precisely when the condition \eqref{stabilityRuledSurfaces} is satisfied.
\end{rmk}
\begin{rmk} Suppose equality holds instead in our ``stability condition" \eqref{stabilityRuledSurfaces},
\begin{equation*} 
\lp 1 + \lp k_1 + k_2 \rp^2 \rp = x \lp 1 + \lp k_1 -k_2 \rp^2 \rp.
\end{equation*}
A direct computation then shows that the quantity $t^2 + C'$ vanishes at the endpoint $t = 1/x-1$.
By our explicit formula \eqref{signSolutions} we see that the function $H_-(t)$ is smooth on the interval $(1/x-1, 1/x+1]$ and extends to a $C^{1/2}$ function on its closure. Thus, for fixed background $\omega$, we obtain a corresponding solution to the dHYM equation which is smooth on $X \setminus E_{\infty}$ and extends to a form with $C^{1/2}$ coefficients on $X$. This should be compared with a result of Takahashi \cite{takaCollapse} which holds for a general compact K\"ahler surface $X$, and states that under suitable assumptions, when the class $[\Omega]$ above is only semipositive, then there exists a solution to the dHYM equation which is smooth on the complement of finitely many holomorphic curves of negative self-intersection and which extends to a closed current on $X$. 
\end{rmk}
\section{Coupled equations}\label{MainThmSmoothSec}

In the previous Section we solved the dHYM equation in suitable integral classes, determining explicitly the Legendre transform of the curvature form $F$ in terms of the K\"ahler metric $\omega$. More precisely, let us assume that the ``stability condition"
\begin{equation*} 
\lp 1 + \lp k_1 + k_2 \rp^2 \rp > x \lp 1 + \lp k_1 -k_2 \rp^2 \rp
\end{equation*}
holds, and let us denote by $F = F(\omega)$ the unique curvature form constructed in the previous Section. 

In this Section we will complete the proof of Theorem \ref{MainThmSmooth} by solving the second equation in \eqref{coupled_dHYM_Intro}. We also establish the real analyticity of our solutions, Proposition \ref{HKextensionProp}.

Recall we are concerned with the equation
\begin{equation}\label{coupledEquationsRE}
s(\omega) -\alpha \re \lp e^{-\ii \hth}\frac{\lp\omega-\ii F \rp^{2}}{\omega^2} \rp  = \hat{s}-\alpha \hat{r},
\end{equation}
where the constants $\hat{s}$ and $\hat{r}$ can be computed as
\begin{equation*}
\hat{s}= 2xs_{\Sigma}+2, \; \hat{r}= \sqrt{\lp1-k_1^2+k_2^2 \rp ^2+4k_1^2}.
\end{equation*}
\begin{lem}\label{RealPartLemma} In terms of the variable $t= 1/x-\TT$ and the function $H(t)$ appearing in \eqref{HdHYM}, we have
\begin{align*}
&\re \lp e^{-\ii \hth}\frac{\lp\omega-\ii F \rp^{2}}{\omega^2} \rp \\&= \cos \hth \lp 1- \frac{H(t) H'(t)}{t} \rp -\sin \hth \lp H'(t) + \frac{H(t)}{t} \rp.
\end{align*}
\end{lem}
\begin{proof}
As in the proof of Proposition \ref{dHYMODEProp}, at a point $(z_0,w_0)$ such that $d \text{log } h (z_0)=0$, we have the global identities 
\begin{align*} 
\nonumber&\frac{1}{2}\re \lp e^{-\ii \hth}\lp \omega-\ii F_g\rp^2 \rp / \ii\frac{dw \wedge d\bar{w}}{|w|^2}\wedge\omega_{\Sigma} =  \\
\nonumber&\cost\lp f'' \frac{1-xf'}{x}+\lp g'-\frac{k_2}{x}\frac{ 1-x^2}{ 1-xf'} -\frac{k_1}{x}+k_1f'\rp \lp g''+k_1f''-k_2f''\frac{ 1-x^2}{ \lp 1-xf' \rp^2} \rp \rp \\
&+\sint \lp \frac{1-xf'}{x} \lp k_2f''\frac{ 1-x^2}{ \lp 1-xf' \rp^2} -g''-k_1f''\rp +f''\lp g'-\frac{k_2}{x}\frac{ 1-x^2}{ 1-xf'} -\frac{k_1}{x}+k_1f'\rp \rp .
\end{align*}
and
\begin{equation*}
   \omega^2 = 2f'' \frac{1-xf'}{x }\ii\frac{dw \wedge d\bar{w}}{|w|^2}\wedge\omega_{\Sigma}.
\end{equation*}
In terms of the variable $t$ and the auxiliary function $H(t)$ we have
\begin{align*}
&\lp f'' \frac{1-xf'}{x}+\lp g'-\frac{k_2}{x}\frac{ 1-x^2}{ 1-xf'} -\frac{k_1}{x}+k_1f'\rp \lp g''+k_1f''-k_2f''\frac{ 1-x^2}{ \lp 1-xf' \rp^2} \rp \rp \\
& = 1- \frac{H(t) H'(t)}{t}, 
\end{align*}
respectively
\begin{align*}
&\lp \frac{1-xf'}{x} \lp k_2f''\frac{ 1-x^2}{ \lp 1-xf' \rp^2} -g''-k_1f''\rp +f''\lp g'-\frac{k_2}{x}\frac{ 1-x^2}{ 1-xf'} -\frac{k_1}{x}+k_1f'\rp \rp \\
&= - H'(t) - \frac{H(t)}{t},
\end{align*}
from which our claim follows immediately.
\end{proof}
\begin{lem}\label{scalODELem} Equation \eqref{coupledEquationsRE} becomes the ODE for the momentum profile $\phi(t)$ given by
\begin{align*}
&\lp \frac{2s_{\Sigma}}{t}-\frac{1}{t}\lp 2t\phi(t) \rp'' \rp +  2\alpha\frac{\cos \hth}{\sin^2 \hth} -\frac{\alpha}{\sin^3 \hth} \frac{t}{\sqrt{\lp\cot^2\hth+1\rp\lp t^2+C'\rp}}\\&-\frac{\alpha}{t\sin \hth}{\sqrt{\lp\cot^2\hth+1\rp\lp t^2+C'\rp}}=\hat{s}-\alpha \hat{r},
\end{align*}
with the boundary conditions
\begin{equation*} 
\lim_{t\to \frac{1}{x} \pm 1} \phi (t) = 0, \quad \lim_{t\to \frac{1}{x} \pm 1} \phi' (t) = \mp 1 .
\end{equation*}
\end{lem}
\begin{proof}
By a standard computation, the scalar curvature of $\omega$ can be expressed in terms of the variable $\tau$ as
\begin{equation*}
    s(\omega) = \frac{2s_{\Sigma}x}{1-x\tau}-\frac{x}{1-x\tau}\lp  2\phi(\tau)\frac{1-x\tau}{x}\rp'',
\end{equation*}
with $\phi(\tau)$ satisfying \eqref{boundaryConditionsT}.
After the affine change of variable $t= 1/x-\TT$, our claim follows directly from Lemma \ref{RealPartLemma} and the explicit formula \eqref{signSolutions} for $H(t)$.
\end{proof}
Setting $\psi(t) = 2t\phi(t)$, we obtain the ODE
\begin{align}\label{ODEscalar}
\nonumber\psi''(t)
&=\nonumber \lp 2\alpha\frac{\cos \hth}{\sin^2 \hth}-\hat{s}+\alpha \hat{r}\rp t-\frac{\alpha}{\sin \hth}{\sqrt{\lp\cot^2\hth+1\rp\lp t^2+C'\rp}}\\ 
& -\frac{\alpha}{\sin^3 \hth}\frac{t^2}{\sqrt{\lp\cot^2\hth+1\rp\lp t^2+C'\rp}}+2s_{\Sigma}
\end{align}
with the boundary conditions
\begin{equation}\label{boundaryConditionsPsi}
\lim_{t\to \frac{1}{x} \pm 1} \psi (t) = 0, \quad \lim_{t\to \frac{1}{x} \pm 1} \psi' (t) = \mp 2 \lp \frac{1}{x} \pm 1 \rp,
\end{equation}
and the positivity condition
\begin{equation}\label{positivityConditionPsi}
\psi(t) > 0, \quad \frac{1}{x}-1<t <\frac{1}{x}+1.
\end{equation}
By integrating twice, we get the general solution of \eqref{ODEscalar} with integration constants $d_0,d_1$
\begin{align}\label{ODEscalarGenSol}
\nonumber\psi(t)=&s_{\Sigma} t^2 +\lp \frac{\alpha}{3}\frac{\cos \hth}{\sin^2 \hth} -\frac{\hat{s}-\alpha \hat{r}}{6} \rp t^3-\frac{\alpha}{3}\sin \hth \lp \lp\cot^2\hth+1\rp\lp t^2+C'\rp\rp^{\frac{3}{2}}\\ &+ d_0 +d_1 t,
\end{align}
which satisfies \eqref{boundaryConditionsPsi} if and only if we set
\begin{align*}
&d_0=-\frac{\lp -2 + s_{\Sigma}x \rp \lp -3 - 3 k_1^2 - 2 k_1 k_2 - 3 k_2^2 + 3 \lp 1 + \lp k_1 - k_2\rp^2\rp x^2 \rp}{3 \lp 1 + \lp k_1 - k_2\rp^2\rp x^3},\\
&d_1 = -\frac{\lp-2 \lp 1 + k_1^2 + k_2^2\rp + \lp 1 + \lp k_1 - k_2\rp^2\rp s_{\Sigma} x\rp \lp -1 + x^2 \rp}{4 k_1 k_2 x^2},\\
& \alpha = \frac{\sqrt{4 k_1^2 + \lp 1 - k_1^2 + k_2^2\rp^2}}{2 \lp 1 + \lp k_1 - k_2\rp^2\rp k_2^2}\lp -2 + s_{\Sigma}x \rp.
\end{align*}
In order to check the positivity condition \eqref{positivityConditionPsi}, we observe that
\begin{equation}\label{psiProp1}
\frac{d^4\psi}{dt^4}= -\frac{3}{4}\alpha \lp \frac{C'}{k_1} \rp^2 \lp t^2+C'\rp^{-\frac{5}{2}} > 0.
\end{equation}
Moreover, setting $t_-=1/x-1$ and $t_+=1/x+1$, we get 
\begin{align}\label{psiProp2}
&\nonumber\psi''(t_-)-\psi''(t_+)\\&=4\frac{-3\lp 1 +\lp k_1+k_2 \rp^2 \rp^2 + \lp 1+\lp k_1-k_2 \rp^2 \rp^2 x^2 \lp x+s_{\Sigma}\rp}{-\lp 1 +\lp k_1+k_2 \rp^2 \rp^2 + \lp 1+\lp k_1-k_2 \rp^2 \rp^2 x^2} > 0, 
\end{align}
since $s_{\Sigma}+x < 3$. 
Thus $\psi''$ is a convex function defined on the interval $\left[ t_- , t_+ \right]$, such that $\psi''(t_-) > \psi''(t_+)$, and this, toghether with \eqref{boundaryConditionsPsi}, implies the positivity condition \eqref{positivityConditionPsi}. 

Finally let us note that if equality holds in our ``stability condition",
\begin{equation*} 
\lp 1 + \lp k_1 + k_2 \rp^2 \rp = x \lp 1 + \lp k_1 -k_2 \rp^2 \rp
\end{equation*}
then the quantity $t^2 + C'$ vanishes at the endpoint $t = 1/x-1$ and by our explicit formulae \eqref{signSolutions}, \eqref{ODEscalarGenSol} we obtain a solution $\omega$, $F$ which is smooth on $X\setminus E_{\infty}$, and such that $F$ extends to a form with $C^{1/2}$ coefficients on $X$, while $\omega$ extends with $C^{1, 1/2}$ coefficients. This completes the proof of Theorem \ref{MainThmSmooth}.\\
\begin{rmk}\label{unifPosRmk} As we will be interested in the small and large limits of the coupled equations, we point out that \eqref{psiProp1} and \eqref{psiProp2} hold uniformly as the scaling parameter $\alpha' \to 0$ and, provided the ``stability condition"
\begin{equation*}
 ( k_1 + k_2 )^2 > x  ( k_1 -k_2 )^2 
\end{equation*}
is satisfied, also for $\alpha' \to \infty$.
\end{rmk}

We can now prove Proposition \ref{HKextensionProp}. We first claim that the K\"ahler form $\omega$ constructed above is real analytic. Recall $\omega$ is obtained by the momentum construction \eqref{ansatzMetric}, 
\begin{equation*}
\omega = \frac{p^*\omega_{\Sigma}}{x}+\ii \del \delbar {f}(s),
\end{equation*}
for a suitable convex function $f\!:\R \to \R$, where we have $s =  \log |w|^2 + \log h(z)$ with respect to bundle adapted holomorphic coordinates $(z, w)$. The hyperbolic metric $\omega_{\Sigma}$ is real analytic, so we can choose a local holomorphic coordinate $z$ such that its coefficients are real analytic. On the other hand the real function $h(z)$ satisfies $-\ii\del_z \delbar_z \log h(z)  = \omega_{\Sigma}$, with the same choice of local coordinate, and so it is also real analytic. So our claim follows if we can show that the function $f\!:\R \to \R$ is real analytic. But $f$ is related to the momentum profile $\phi$ by the ODE
\begin{equation*}
f''(s) = \phi(\tau) = \phi(f'(s)),
\end{equation*}
and the momentum profile $\phi(\tau)$ or our solution is clearly a real analytic function of the variable $\tau \in (-1,1)$ by \eqref{ODEscalarGenSol}. Thus $f(s)$ is real analytic and our claim on $\omega$ follows. In order to see that the curvature form $F$ is also real analytic, recall that it is given by our ansatz \eqref{ansatzCurvature}, $F = F_0 + \ii \del \delbar g(s)$, and that the dHYM equation satisfied by $F$ can be expressed in terms of $g(s)$ as the vanishing of the right hand side of the expression \eqref{dHYMsVariable}. Thus, the real analitycity of $g(s)$ follows from that of $f(s)$. 
\section{Conical singularities}\label{MainThmConicSec}
In the present Section we prove Theorem \ref{MainThmConic}. This extends our existence result Theorem \ref{MainThmSmooth} to allow a K\"ahler form $\omega$ with conical singularities. Our main motivation for this extension is describing examples of solutions to the coupled equations \eqref{coupled_dHYM_Intro} with positive coupling constant $\alpha > 0$. 

We consider again K\"ahler forms $\omega$ given by the momentum construction \eqref{ansatzMetric}, 
\begin{equation*} 
\omega = \frac{p^*\omega_{\Sigma}}{x}+\ii \del \delbar {f}(s),
\end{equation*}
with momentum profile $\phi(\TT) > 0$ defined on the interval $(-1, 1)$.
\begin{lem}\label{ConicalBoundaryLemma} The K\"ahler form $\omega$ extends to a form with conical singularities on $X$, with cone angle $2\pi\beta_0$ along $E_0$, respectively $2\pi\beta_{\infty}$ along $E_{\infty}$, iff the momentum profile satisfies the boundary conditions
\begin{equation*} 
\lim_{\TT\to \pm 1} \phi (\TT) = 0, \, \lim_{\TT\to - 1} \phi' (\TT) = \beta_0,\, \lim_{\TT\to 1} \phi' (\TT) = -\beta_{\infty}.
\end{equation*}  
\end{lem} 
\begin{proof}
For any open neighborhood $U \subset X$, in term of the bundle adapted coordinates $(z,w)$, $E_0 \cap U = \left\{ w=0\right\}$. We assume that, near $r=|w|=0$, $f''$ has the form
\begin{equation*}
    f''(s) = c_0 r^{2\beta_0} + A(r) 
\end{equation*}
with $c_0 \neq 0$ and $A(r)=o(r^{2\beta_0})$.
Then we have
\begin{align*}
& \omega_{z\overline{z}}=  \lp \frac{1-xf'}{x}\rp\omega_{\Sigma,z\overline{z}}+\ii f''\frac{\del h \delbar h}{h^2}=O(1),\\
& \omega_{w\overline{z}}= -\sqrt{-1} \frac{1}{w}f''(s)\delbar h = O(r^{2\beta_0-1}),\\
&  \omega_{z\overline{w}} = \sqrt{-1} \frac{1}{\overline{w}}f''(s)\del h = O(r^{2\beta_0-1}),\\
&  \omega_{w\overline{w}}=\ii r^{2\beta_0-2}\lp 1+A(r)/r^{2\beta_0}\rp,
\end{align*}
hence the metric $\omega$ given by the momentum construction has a conical singularity along $E_0$ of angle $2\pi\beta_0$.
Since $d/ds = \frac{r}{2}d/dr$, $f''(s) = \phi(\tau)$ and $ f'''(s)=\phi(\tau)\phi' (\tau)$, this implies
\begin{equation*}
    \lim_{\tau \to -1} \phi(\tau)=0
\end{equation*}
and
\begin{equation*}
    \lim_{\tau \to -1} \phi'(\tau)=\beta_0.
\end{equation*}
To proof for $E_{\infty}$ is the same up to a change of variable.
\end{proof}
As in the previous Section, it is convenient to consider the reparametrisation
\begin{equation*}
x = \frac{k}{k+ k'}
\end{equation*}
for $k'>0$. Similarly, we introduce the $(1,1)$-forms 
\begin{align*}
& \beta = \frac{x^2}{\left( 1-xf'(s)\right)^2}\left(\frac{1-x {f}'(s)}{x}\omega_{\Sigma} - \ii {f}''(s)\frac{dw \wedge d\bar{w}}{|w|^2} \right),\\
& F = k_1 \omega + \frac{1-x^2}{x^2}k_2 \beta, 
\end{align*}
as well as the ansatz, extending the momentum construction 
\begin{equation*}
F_g = F + \ii \del \delbar g(s).
\end{equation*}
We also denote by $\nu(\tau)$ the image of $g'(s)$ under the Legendre transform diffeomorphism relative to $f(s)$. The proof of the following result is almost identical to the smooth case and we leave it to the reader.
\begin{lem}\label{ConicalMetricLemma} A K\"ahler form $\omega$ with conical singularities as above is a closed $(1,1)$-current on $X$, with cohomology class
\begin{equation*}
[\omega] = 2\pi[ 2 E_0 + k' C].
\end{equation*}  
Similarly, $F$ is a closed $(1, 1)$-current on $X$ with cohomology class
\begin{equation*}
[F] = 2\pi[2 (k_1 - k_2)  E_0  + (2 k k_2 + k' (k_1  + k_2 )  C].
\end{equation*}
Moreover, $\ii\del \delbar g(s)$ extends to a closed $(1,1)$-current on $X$, which has vanishing cohomology class iff $\nu(\tau)$ satisfies the boundary conditions
\begin{equation*} 
\lim_ {\TT \to \pm 1} \nu(\TT) = 0.
\end{equation*}
\end{lem}
We are now in a position to complete the proof of Theorem \ref{MainThmConic}. Let us first note that, precisely as in the proof of Theorem \ref{MainThmSmooth}, under the momentum construction the dHYM equation for $\omega$ and $F$ becomes the ODE \eqref{HdHYM}, together with the boundary conditions \eqref{bdryH}. By Lemma \ref{ConicalMetricLemma}, the cone angles do not play a role in this reduction. It follows that the second of our coupled equations \eqref{coupled_dHYM_Intro} also reduces to the same ODE \eqref{ODEscalar} for a single function $\psi(t) > 0$ of the variable
\begin{equation*}
t = 1/x-\TT
\end{equation*}
appearing in the proof of Theorem \ref{MainThmSmooth}. By Lemma \ref{ConicalBoundaryLemma}, the boundary conditions corresponding to general cone angles $\beta_0$, $\beta_{\infty}$ are 
\begin{equation*} 
\lim_{t\to \pm 1} \psi (t) = 0,\,  \lim_{t\to \frac{1}{x} + 1} \psi' (t) = - 2 \beta_0 \lp \frac{1}{x} + 1 \rp,\,\lim_{t\to \frac{1}{x} - 1} \psi' (t) =  2 \beta_{\infty}\lp \frac{1}{x} - 1 \rp.
\end{equation*}
However, as \eqref{ODEscalar} is second order ODE, this problem is overdetermined. If we consider the general solution \eqref{ODEscalarGenSol} and impose the boundary condition
\begin{equation*}
\lim_{t\to \frac{1}{x} + 1} \psi' (t) = - 2 \beta_0 \lp \frac{1}{x} + 1 \rp
\end{equation*}
corresponding to a cone angle $2\pi \beta_0$ along $E_0$, we find that the integration constant $d_1$ can be expressed in terms of $\beta_0$ and the coupling constant $\alpha$ as
\begin{align}\label{linearTermCoeff}
\nonumber d_1 &= \frac{(x+1) (2 k_1 (x (-2 \beta_0 +s_{\Sigma} (x-1)+1)+1)}{2 k_1 x^2}\\
    & -\alpha\frac{ k_2 (x-1)
   \sqrt{k_1^4-2 k_1^2
    (k_2^2-1)+(k_2^2+1)^2})}{2 k_1 x^2}.
\end{align}
Similarly, imposing the condition
\begin{equation*} 
\lim_{t\to \frac{1}{x} + 1} \psi (t) = 0 
\end{equation*}
and using our expression for $d_1$ gives the relation
\begin{align*}
d_0 &= 4 \alpha\frac{ k_2^2 \left(x^3
   (k_1-k_2)^2+(k_1+k_2)^2+x^3+1\right)}{3 x^3 \sqrt{\left(-k_1^2+k_2^2+1\right)^2+4 k_1^2}}\\
   &-\frac{(x+1)^2
   \sqrt{k_1^4-2 k_1^2
   \left(k_2^2-1\right)+\left(k_2^2+1\right)^2} (x (-6 \beta_0 +s_{\Sigma} (2
   x-1)+2)+2)}{3 x^3 \sqrt{\left(-k_1^2+k_2^2+1\right)^2+4 k_1^2}}.
\end{align*}
Further, imposing the condition 
\begin{equation*} 
\lim_{t\to \frac{1}{x} - 1} \psi (t) = 0 
\end{equation*}
and using our expressions for $d_0$, $d_1$ determines the coupling constant uniquely as
\begin{equation}\label{ConicalCouplingConstant}
\alpha = \frac{\sqrt{\left(-k_1^2+k_2^2+1\right)^2+4 k_1^2} \left(s_{\Sigma} x^2-3 \beta_0
    (x+1)+x+3\right)}{2 k_2^2 x \left((k_1-k_2)^2+1\right)}.
\end{equation}
We can now compute directly that a solution $\psi(t)$ corresponding to a cone angle $2\pi\beta_0$ along $E_0$ satisfies
\begin{equation*}
\lim_{t\to \frac{1}{x} - 1} \psi' (t) = -\frac{2 (\beta_0 +\beta_0 x-2)}{x} = 2  \frac{-2 + \beta_{0}(1+x)}{-1+x}\lp \frac{1}{x} - 1 \rp, 
\end{equation*}
which yields a cone angle $2\pi\beta_{\infty}$ along $E_{\infty}$, with
\begin{equation*}
\beta_{\infty} = \frac{-2 + \beta_{0}(1+x)}{-1+x}.
\end{equation*}
In order to prove the positivity of $\psi(t)$, we consider again \eqref{psiProp1}, with the coupling constant $\alpha$ given by \eqref{ConicalCouplingConstant}. When 
\begin{equation*}
    3\lp 1+x \rp\beta_0 -3 > x\lp1+ s_{\Sigma}x\rp,
\end{equation*}
we construct solutions for $\alpha<0$ and hence $\frac{d^4\psi}{dt^4}>0$. 
Moreover
\begin{align*}
&\psi''(t_-)-\psi''(t_+)\\&=4\frac{\lp 3 \lp1 + x\rp \beta_0 - 3\rp\lp 1 + \lp k_1 + k_2\rp^2\rp^2 - x^2\lp  1 + \lp k_1 - k_2\rp^2\rp^2\lp s_{\Sigma}x^2 + x\rp}{\lp 1 +\lp k_1+k_2 \rp^2 \rp^2x - \lp 1+\lp k_1-k_2 \rp^2 \rp^2 x^3} > 0
\end{align*}
and we can use essentially the same argument given in the proof of Theorem \ref{MainThmSmooth}. When $\alpha>0$, an explicit analysis of the momentum profile is more complicated and the positivity of $\psi(t)$ is best checked with the assistance of a numerical software package (see Figure \ref{Figure1}). This completes the proof of Theorem \ref{MainThmConic}.
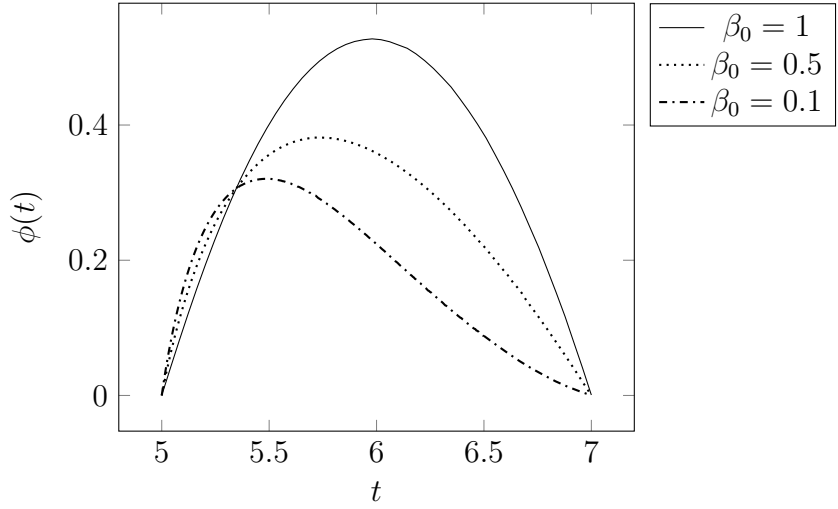
\begin{figure}[h!] 
    \centering
\begin{tikzpicture}
legend pos=south east,
    legend entries={Entry 1,Entry 2,Entry 3}
    ]
\begin{axis}[domain=5:9, samples=100,xlabel=$t$,ylabel=$\phi(t)$, legend pos=outer north east, legend entries={$\beta_0=1$,$\beta_0=0.5$,$\beta_0=0.1$}]

\addplot[color=black, domain=5:7, samples=50]{(-158 + (455*x)/12 + 2*x^2 - (47*x^3)/72 + 
 (5/72)*2.23607*(-(124/5) + x^2)^(3/2))/(2*x)};
 
\addplot[color=black, domain=5:7, samples=100, thick, dotted]{(102.4 - 46.9583*x + 2*x^2 + 0.659722*x^3 - 
 0.822997*(-(124/5) + x^2)^(3/2))/(2*x)};
 
\addplot [color=black, domain=5:6.5, samples=300,thick,dash dot] {(310.72 - 114.858*x + 2*x^2 + 1.70972*x^3 - 
 1.60562*(-(124/5) + x^2)^(3/2))/(2*x)};
 \addplot[color=black, domain=6.5:7, samples=10,thick,dash dot]{(310.72 - 114.858*x + 2*x^2 + 1.70972*x^3 - 
 1.60562*(-(124/5) + x^2)^(3/2))/(2*x)};

\end{axis}
\end{tikzpicture}
    \caption{The momentum profile $\phi(t)$ of the solution when $k_2=-k_1=1$, $h=0$ and $x=1/6$.}
    \label{Figure1}
\end{figure}
\section{Twisted K\"ahler-Einstein equation}
This Section is devoted to the proof of Proposition \ref{tKEProp}, which states explicitly when the equation in \eqref{coupled_dHYM_Intro} involving the scalar curvature of $\omega$ reduces to a twisted K\"ahler-Einstein equation. For a general complex surface, we should require that
\begin{equation}\label{CohomologyTwistedKE}
[\operatorname{Ric}(\omega)] + \frac{\alpha}{2\sin\hth} [F] = \frac{\hat{s} - \alpha\hat{r}}{4} [\omega],
\end{equation}
and we will make this condition explicit in our current setting. 

\begin{lem}\label{ConicalRicciLemma}
For any K\"ahler form $\omega$ on $X$ given by the momentum construction, with cone angle $2\pi\beta_0$ along $E_0$, respectively $2\pi\beta_{\infty}$ along $E_{\infty}$, the cohomology class of $\operatorname{Ric}(\omega)$ is given by
\begin{equation*}
    \left[ \frac{\operatorname{Ric}(\omega)}{2\pi}\right] = \left(\beta_0 + \beta_{\infty} \right) \left[ E_0\right]+ \left(2\left( 1-h\right) -k\beta_{\infty}\right)\left[ C\right].
\end{equation*}
\end{lem}
\begin{proof}
We recall that 
\begin{align*}
    \operatorname{Ric}(\omega) &= -\sqrt{-1}\partial \overline{\partial} \log \det \omega\\
    &= -\sqrt{-1}\partial \overline{\partial} \log \det \left( \frac{2}{|w|^2}\left( \frac{1}{x}-f'\right)f''\omega_{\Sigma}\right)\\
    &=-\sqrt{-1}\partial \overline{\partial} \log \det \left( \left( \frac{1}{x}-f'\right)f''\omega_{\Sigma}\right),
\end{align*}
hence, by a straightforward calculation, we get
\begin{equation*}
    -\sqrt{-1}\partial_z \partial_{\overline{z}} \log \det \left( \left( \frac{1}{x}-f'\right)f''\omega_{\Sigma}\right) = \left( \frac{f'''}{f''}+ \rho_{\Sigma}-\frac{x}{\left(1-xf' \right)}f''\right)\omega_{\Sigma}
\end{equation*}
and 
\begin{equation*}
    -\sqrt{-1}\partial_w \partial_{\overline{w}} \log \det \left( \left( \frac{1}{x}-f'\right)f''\omega_{\Sigma}\right) = \sqrt{-1} \frac{1}{|w|^2}\frac{d}{ds}\left( \frac{x}{\left(1-xf' \right)}f''-\frac{f'''}{f''} \right)dw\wedge d\overline{w}.
\end{equation*}
Using the identities $f'''(s)/f''(s)=\phi'(\tau)$, $f''(s)=\phi(\tau)$ and the boundary conditions required for the momentum profile and its derivative, we compute
\begin{align*}
    \int_{E_0} \operatorname{Ric}(\omega) &= \left( \phi'(-1) +\frac{2(1-h)}{k} -\frac{\phi(-1)}{x^{-1}-1}\right)\int_{\Sigma}\omega_{\Sigma}dz\wedge d\overline{z}\\
    &=2\pi \left( 2(1-h)+k\beta_0\right)
\end{align*}
and
\begin{align*}
    \int_C \operatorname{Ric}(\omega) &= \int_{\mathbb{C} \setminus0} \sqrt{-1}\frac{dw \wedge d \overline{w}}{|w|^2}\frac{d}{ds} \left( \frac{x}{\left(1-xf' \right)}f'' -\frac{f'''}{f''}\right)\\
    &= \int_{-\infty}^{\infty} \int_0^{2\pi} \frac{d}{dr}\left( \frac{x}{\left(1-xf' \right)}f'' -\frac{f'''}{f''}\right)dr\wedge d\theta\\
    &=2\pi \left(\beta_0 + \beta_{\infty} \right),
\end{align*}
so our claim follows directly from \eqref{explicitCoordinates}.
\end{proof}
For the following computations, it is convenient to introduce the quantity 
\begin{equation*}
    \Gamma =\frac{3+x+s_{\Sigma}x^2-3(1+x)\beta_0}{x}= 4-6\beta_0 +3\frac{k'}{k}(1-\beta_0)+2\frac{1-h}{k+k'}.
\end{equation*}
Using Lemma \ref{ConicalMetricLemma} and Lemma \ref{ConicalRicciLemma}, we can then rephrase the general condition \eqref{CohomologyTwistedKE} as the system of equations
\begin{align}\label{SystemTwistedKE}
\begin{cases}
&\frac{1+\left(k_1+k_2 \right)^2}{2k_1k_2}\Gamma=2+4\frac{1-h}{k+k'}-2\left(\beta_0+\beta_{\infty} \right), \\
&\frac{1+\left(k_1+k_2 \right)^2}{2k_1k_2}\Gamma\left(\frac{k'}{2}+k \right)=2\left( 1-h \right)\frac{2k+k'}{k+k'}-2k\beta_{\infty}-k'.
\end{cases}
\end{align}
Notice that the two equations in \eqref{SystemTwistedKE} actually coincide when
\begin{equation*}
    -2\left(1-h \right)\frac{2k+k'}{k+k'}+2k\beta_{\infty}+k'=\left(2k+k' \right)\left(\beta_0+\beta_{\infty}-1-2\frac{1-h}{k+k'} \right),
\end{equation*}
or, equivalently,
\begin{equation}\label{CompatibilitySystem}
    2\left(k'+k\right) = \left(2k+k' \right)\beta_0+k'\beta_{\infty}.
\end{equation}
Recall however that in order to have solutions to our equations in the momentum construction the cone angle $2\pi \beta_0$ is not arbitrary but must satisfy
\begin{equation*}
    \beta_{\infty} = \frac{-2+\beta_0(1+x)}{(-1+x)},
\end{equation*}
in which case \eqref{CompatibilitySystem} holds automatically. Then the general condition \eqref{CohomologyTwistedKE} corresponds to
\begin{align}\label{CohomologyTwistedKE2}
\nonumber    \frac{1+\left(k_1+k_2 \right)^2}{2k_1k_2}\Gamma&=2+4\frac{1-h}{k+k'}-2\left(\beta_0+\beta_{\infty} \right)\\&=2\left(2\frac{1-h}{k+k'}+2\frac{k}{k'}\left(\beta_0-1 \right)-1 \right).
\end{align}
In order to show that this coincides with the condition \eqref{tKEPropCond} spelled out in Proposition \ref{tKEProp}, we rewrite the latter as
\begin{equation*}
    \left( -1+x\right)\left( 1+\left(k_1+k_2\right)^2\right)\Gamma -4k_1k_2\left(1+s_{\Sigma}-x\left(-1+s_{\Sigma}+2\beta_0 \right) \right)=0,
\end{equation*}
which implies
\begin{align*}
    \frac{ 1+\left(k_1+k_2\right)^2}{2k_1k_2}\Gamma &= 2 \frac{1}{-1+x}\left(1+s_{\Sigma}-x\left(-1+s_{\Sigma}+2\beta_0 \right)\right)\\
    &=2\left(s_{\Sigma}x+2\beta_0 \frac{x}{1-x}+\frac{1+x}{-1+x} \right)\\
    &= 2\left(s_{\Sigma}x+2\frac{x}{1-x}\left(\beta_0-1 \right)-1 \right)\\
    &=2\left(2\frac{1-h}{k+k'}+2\frac{k}{k'}\left(\beta_0-1 \right)-1 \right).
\end{align*}
Reading these identities backwards shows that the two conditions \eqref{tKEPropCond}, \eqref{CohomologyTwistedKE2} are indeed equivalent.

It remains to establish the second claim of Proposition \ref{tKEProp}, namely that the condition \eqref{CohomologyTwistedKE2} actually holds for infinitely many solutions of the system \eqref{coupled_dHYM_Intro}. It is convenient to rewrite \eqref{CohomologyTwistedKE2} in the form
\begin{equation}\label{CohomologyTwistedKE4}
    F(k_1,k_2)=H(k,k',h,\beta_0),
\end{equation}
with 
\begin{equation*}
    F(k_1,k_2)=\frac{1+\left(k_1+k_2 \right)^2}{2k_1k_2}
\end{equation*}
and
\begin{align*}
   H(k,k',h,\beta_0)&=\frac{2}{\Gamma}\left(2\frac{1-h}{k+k'} +1-\beta_0-\beta_{\infty}\right)\\
    &=2\frac{2\frac{1-h}{k+k'}+2(\beta_0-1)\frac{k}{k'}-1}{2\frac{1-h}{k+k'} +3\frac{k'}{k}(1-\beta_0)+4-6\beta_0}.
\end{align*}
We assume that $k_2<0$, so the stability condition \eqref{stabilityRuledSurfaces} is automatically satisfied, and the system \eqref{coupled_dHYM_Intro} is solvable. We observe that, under this assumption, the l.h.s. of \eqref{CohomologyTwistedKE4} satisfies
\begin{equation*}
   F(k_1,k_2)>2.
\end{equation*}
On the other hand $H(k,k',h,\beta)$, as a function of the single variable $\beta$, has a vertical asymptote at 
\begin{equation*}
    \overline{\beta}=\frac{\frac{4}{3}k+k'+\frac{2(1-h)k}{3(k+k')}}{k'+2k}
\end{equation*}
and it is easy to check that $0< \overline{\beta}<1$, for $k'>M(k,h)>0$.
Moreover, at $\beta=1$ we have
\begin{equation*}
    0<H(k,k',h,1)= \frac{k+k'+2\lp h-1\rp}{k+k'+h-1}<2,
\end{equation*}
and
\begin{equation*}
    \frac{d}{d\beta}H(k,k',h,1)= \frac{4 k}{k' \lp -2 + \frac{2 \lp 1 - h\rp}{k + k'} \rp}+\frac{2\lp6+\frac{3k}{k'} \rp \lp -1+ \frac{2 \lp 1 - h\rp}{k + k'}\rp}{\lp -2 + \frac{2 \lp 1 - h\rp}{k + k'} \rp^2}<0,
\end{equation*}
(Figure \ref{Figure2} shows the graph of  $H\lp\beta \rp$ for $k=k'=1$ and $h=6$).
\begin{figure}[h!]
    \centering
\begin{tikzpicture}
\begin{axis}[
    axis lines = left,
    xlabel = $\beta$,
    xtick={2/9},
    xticklabels={$\overline{\beta}$},
    extra x ticks={0,0.5,1},
    extra x tick labels={0,0.5,1},
    ylabel = {},
    ytick={1.71},
    yticklabels={$H(\beta=1) \approx$1.71},
    extra y ticks={-20,20},
    extra y tick labels={-20,20},
]
%Below the red parabola is defined
\addplot [
    domain=0:(2/9-0.05), 
    samples=50, 
    color=black,
]
{(-16+4*x)/(2-9*x)};
\addplot [
    domain=(2/9+0.05):1, 
    samples=50, 
    color=black,
]
{(-16+4*x)/(2-9*x)};
 \addplot[dashed, black, samples=2,domain=0:1, ] {1.71};
\draw[dashed] ({axis cs:2/9,0}|-{rel axis cs:0,0}) -- ({axis cs:2/9,0}|-{rel axis cs:0,1});
\end{axis}
\end{tikzpicture}
    \caption{$H(k,k',h,\beta)$ for $k=k'=1$ and $h=6.$}
    \label{Figure2}
\end{figure}
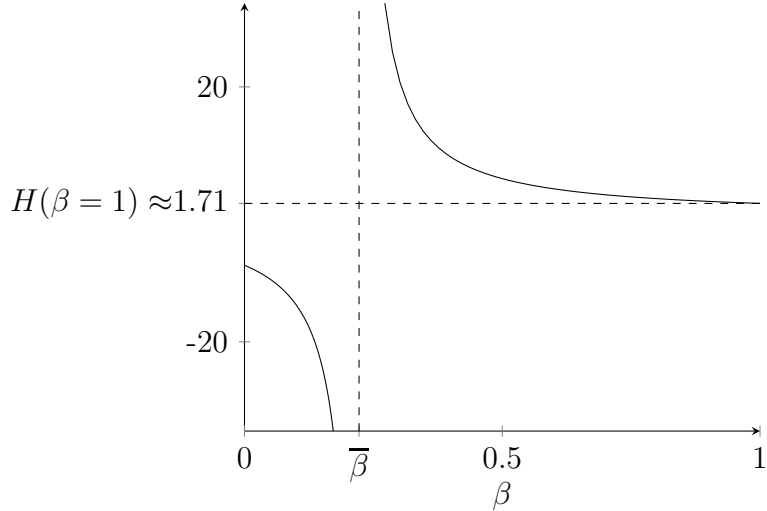
This implies that
\begin{equation*}
    F(k_1,k_2) \in (2,\infty) \subset \operatorname{im} H|_{\beta \in \lp 0,1\rp},
\end{equation*}
which completes the proof of Proposition \ref{tKEProp}.

\begin{rmk} A direct computation using \eqref{linearTermCoeff}, shows that the condition \eqref{tKEPropCond} holds precisely when the coefficient of the linear term $d_1$ in $\psi(t)$ vanishes, i.e. a solution $\omega$ is twisted K\"ahler-Einstein precisely when the linear term is missing from the momentum profile.
\end{rmk}
\section{Large and small radius limits}\label{LargeRadiusSec}
Let us first prove Theorem \ref{LargeRadiusThm}. As we already observed, the ``slope parameter" $\alpha'$ appears in the coupled equations \eqref{coupled_dHYM_Intro_scale} simply as a scale factor for the curvature form $F$. In other words, a pair $\omega, F$ solves \eqref{coupled_dHYM_Intro_scale} iff the pair $\omega, \alpha' F$ solves \eqref{coupled_dHYM_Intro}: the cohomology parameters are simply rescaled $(k_1, k_2)\mapsto (\alpha' k_1, \alpha' k_2)$. Thus, according to Theorem \ref{MainThmSmooth}, there exists a (unique) solution of \eqref{coupled_dHYM_Intro_scale} given by the momentum construction iff the ``stability condition"
\begin{equation*} 
  1 + (\alpha')^2\lp k_1 + k_2 \rp^2   > x \lp 1 + (\alpha')^2\lp k_1 -k_2 \rp^2 \rp
\end{equation*}
holds. Since $x < 1$ by construction, this inequality holds for all sufficiently small $\alpha'$, depending only on $k_1, k_2$ and $k'$. Let us write $\omega_{\alpha'}, F_{\alpha'}$ for the corresponding family of solutions. 

The attached function $H_{\alpha'}(t) = H_{-}(t)$ appearing in \eqref{HdHYM} is also obtained from \eqref{signSolutions} simply by rescaling $(k_1, k_2)\mapsto (\alpha' k_1, \alpha' k_2)$, and so it can be computed explicitly as
\begin{align*}
\nonumber&H_{\alpha'}(t) = t \cott_{\alpha'} - \sqrt{\lp\cot^2\hth_{\alpha'}+1\rp\lp t^2+C'_{\alpha'}\rp},\\
\nonumber&e^{\ii \hth_{\alpha'}}= \frac{\left(1-(\alpha')^2k_1^2+(\alpha')^2k_2^2 -2\ii  \alpha'k_1 \right)}{\sqrt{\left(1-(\alpha')^2 k_1^2+(\alpha')^2k_2^2\right)^2 +\left(2(\alpha')^2 k_1 \right)^2}},\\
\end{align*}
\begin{align}\label{explicitH}
\nonumber&\cott_{\alpha'} = -\frac{ - (\alpha')^2 k_1^2 +(\alpha')^2 k_2^2 +1 }{2 \alpha' k_1  },\\
&C'_{\alpha} =4 (\alpha')^2 k_1 k_2 
   \left(\frac{1}{x^2 \left((\alpha' k_1  - \alpha' k_2
   )^2+1\right)}-\frac{1}{(\alpha' k_1 + \alpha' k_2 
   )^2+1}\right).
\end{align}
By elementary computations using these explicit formulae, recalling that we also have $k_1<0$, we find
\begin{equation}\label{Hexpansion}
H_{\alpha'}(t) = \left(k_1 t + \frac{k_2}{t} \left(-1 + \frac{1}{x^2}\right)\right)\alpha' + (\alpha')^2 R(\alpha', t)
\end{equation} 
for some function $R(\alpha',t)$, smooth up to $\alpha' = 0$. 

As a first consequence, we can show that the sequence of K\"ahler forms $\omega_{\alpha'}$ converges smoothly to a K\"ahler form $\omega$ as $\alpha' \to 0$. It will be enough to show the smooth convergence of the momentum profiles $\phi_{\alpha'}(t)$. According to Lemma \ref{scalODELem} and the subsequent explicit formulae for the coupling constant $\alpha$ and average radius $\hat{r}$, the profile $\phi_{\alpha'}(t)$ is obtained by integrating twice the identity
\begin{align}\label{profileExpansion}
\nonumber\frac{2s_{\Sigma}}{t}-\frac{1}{t}\lp 2t\phi(t) \rp'' &= \alpha \left(\cos \hth \lp 1- \frac{H(t) H'(t)}{t} \rp -\sin \hth \lp H'(t) + \frac{H(t)}{t} \rp\right) \\&+ \hat{s}-\alpha\hat{r},
\end{align}
where all quantities are understood as evaluated at $(\alpha'k_1, \alpha' k_2 )$, and in particular
\begin{align}\label{explicitAlpha}
\nonumber&\alpha = \frac{\sqrt{4 (\alpha')^2 k_1^2 + \lp 1 - (\alpha')^2k_1^2 + (\alpha')^2k_2^2\rp^2}}{2 \lp 1 + \lp \alpha'k_1 - \alpha'k_2\rp^2\rp (\alpha')^2k_2^2}\lp -2 + s_{\Sigma}x \rp\,\\
&\hat{r} = \sqrt{\lp1-(\alpha')^2k_1^2+(\alpha')^2k_2^2 \rp ^2+4(\alpha')^2k_1^2}.
\end{align}
By the latter explicit formulae and \eqref{Hexpansion}, the quantity
\begin{equation*}
\alpha \left(\cos \hth \lp 1- \frac{H(t) H'(t)}{t} \rp -\sin \hth \lp H'(t) + \frac{H(t)}{t} \rp - \hat{r}\right)
\end{equation*}
has a smooth limit as $\alpha' \to 0$, so the same holds for the right hand side of \eqref{profileExpansion} and for the momentum profile $\phi(t) = \phi_{\alpha'}(t)$. The positivity of $ \phi_{\alpha'}(t)$ and its limit for $\alpha' \to 0$ follows from Remark \ref{unifPosRmk}.
We can now show that the curvature forms $F_{\alpha'}$ also converge smoothly as $\alpha'\to 0$. By construction we have $F_{\alpha'} = F_{0,\alpha'} + \ii \del \delbar (\alpha')^{-1} g_{\alpha'}(s)$, where $F_0 = c_1 \omega_{\alpha'} + c_2 \beta_{\alpha'}$ and the potential $g_{\alpha'}(s)$ corresponds to the solution for the parameters $(\alpha' k_1, \alpha'k_2)$ (i.e. for the cohomology class $\alpha'[F_0]$).  By the smooth convergence of the K\"ahler forms $\omega_{\alpha'}$, which we just established, it will be enough to show that the potentials $(\alpha')^{-1} g_{\alpha'}(s)$ converge smoothly. In fact they converge smoothly to the zero potential. Indeed by \eqref{definitionH} and \eqref{Hexpansion} we have
\begin{align*}
(\alpha')^{-1}g_{\alpha'}(s) &= (\alpha')^{-1} \nu_{\alpha'}(\tau)\\
&= (\alpha')^{-1}\left(\alpha' k_1 t+\alpha'\frac{k_2}{t}\frac{1-x^2}{x^2} - H_{\alpha'}(t)\right)\\
&= \alpha' R(\alpha', t) 
\end{align*}
where $R(\alpha', t)$ is smooth in a neighbourhood of $\alpha' = 0$. It follows that we have, smoothly as $\alpha' \to 0$,
\begin{equation*}
F_{\alpha'} \to F_0 = c_1 \omega + c_2 \beta,
\end{equation*}
which is indeed a solution of the HYM equation $\Lambda_{\omega} F = \mu$.

Finally, this allows to write down the equation satisfied by the limit K\"ahler form $\omega$. Recall $\omega_{\alpha'}$ solves the equation
\begin{equation*} 
s(\omega_{\alpha'}) -\alpha_{\alpha'} \re \lp e^{-\ii \hth_{\alpha'}}\frac{\lp\omega_{\alpha'}-\ii \alpha' F_{\alpha'} \rp^{2}}{\omega^2_{\alpha'}} \rp  = \hat{s}-\alpha_{\alpha'} \hat{r}_{\alpha'}.
\end{equation*}
Expanding around $\alpha' = 0$ we find
\begin{align*}
&\re\left(e^{-\ii\hth_{\alpha'}}( \omega_{\alpha'} - \ii \alpha' F_{\alpha'})^2\right)\\& =  \omega^2_{\alpha'} -\left( F_{\alpha'} \wedge F_{\alpha'} - z_1\omega_{\alpha'}\wedge F_{\alpha'} +z_2\omega^2_{\alpha'}\right )(\alpha')^2 + O(\alpha'^{4})  
\end{align*} 
for certain cohomological constants $z_1$, $z_2$. Similarly, 
\begin{align*}
&\alpha_{\alpha'} = \frac{1}{(\alpha')^2} \left(\frac{-2 + s_{\Sigma}x}{2 k_2^2} + O(\alpha')\right),\\
&\hat{r}_{\alpha'} = 1 + O(\alpha').
\end{align*}
Thus, taking the smooth limit as $\alpha' \to 0$, and using our result that the limit curvature form satisfies $\Lambda_{\omega} F = \mu$, we see that $\omega$ satisfies 
\begin{equation*}
s(\omega) + \tilde{\alpha} \Lambda^2_{\omega}(F\wedge F) = c
\end{equation*}
where
\begin{equation*}
\tilde{\alpha} = \frac{-2 + s_{\Sigma}x}{2 k_2^2}\\
\end{equation*}
and $c$ is a cohomological constant. This completes the proof of Theorem \ref{LargeRadiusThm}.\\

The proof of Theorem \ref{SmallRadiusThm} is quite similar. Our assumption 
\begin{equation*}
( k_1 + k_2 )^2 > x  ( k_1 -k_2 )^2
\end{equation*}
implies that, for any $\alpha'>0$, the ``stability condition"
\begin{equation*}
 1+ ( \alpha' k_1 + \alpha' k_2 )^2 > x(1 + ( \alpha' k_1 - \alpha' k_2 ))^2
\end{equation*}
holds. Thus, by Theorem \ref{MainThmSmooth},  the coupled equations \eqref{coupled_dHYM_Intro_scale} are uniquely solvable with the momentum construction. We denote the corresponding solutions by $\omega_{\alpha'}$, $F_{\alpha'}$ as before. By \eqref{explicitH}, as $\alpha' \to \infty$, we have an expansion
\begin{align}\label{HexpansionSmall}
\nonumber H_{\alpha'}(t) &= \frac{\sqrt{\left(k_1^2-k_2^2\right)^2 \left(k_1 k_2
   \left(\frac{4}{x^2 (k_1-k_2)^2}-\frac{4}{(k_1+k_2)^2}\right)+t^2\right)}+k_1^2 t-k_2^2 t}{2 k_1} \alpha' \\
&+ S((\alpha')^{-1}, t)
\end{align}
where $S(y, t)$ is a smooth function near $y = 0$. By this expansion and \eqref{explicitAlpha}, the quantity
\begin{equation*}
\alpha \left(\cos \hth \lp 1- \frac{H(t) H'(t)}{t} \rp -\sin \hth \lp H'(t) + \frac{H(t)}{t} \rp - \hat{r}\right)
\end{equation*}
has a smooth limit as $\alpha' \to \infty$, so the same holds for the right hand side of \eqref{profileExpansion} and for the momentum profile $\phi(t) = \phi_{\alpha'}(t)$. Since we are assuming $( k_1 + k_2 )^2 > x  ( k_1 -k_2 )^2$, $\phi_{\alpha'}(t)$ and its limit satisfy the positivity condition, by Remark \ref{unifPosRmk}. Thus the sequence of K\"ahler forms $\omega_{\alpha'}$ converges smoothly to a K\"ahler form $\omega$ as $\alpha' \to \infty$.

Considering now the curvature forms $F_{\alpha'} = F_{0,\alpha'} + \ii \del \delbar (\alpha')^{-1} g_{\alpha'}(s)$ as before, we find 
\begin{align*}
(\alpha')^{-1}g_{\alpha'}(s) &= (\alpha')^{-1} \nu_{\alpha'}(\tau)\\
&= (\alpha')^{-1}\left(\alpha' k_1 t+\alpha'\frac{k_2}{t}\frac{1-x^2}{x^2} - H_{\alpha'}(t)\right)\\
&= k_1 t + \frac{k_2}{t}\frac{1-x^2}{x^2} - (\alpha')^{-1} H_{\alpha'}(t),  
\end{align*}
where by \eqref{HexpansionSmall} we have the smooth convergence, as $\alpha' \to \infty$,
\begin{equation*}
(\alpha')^{-1} H_{\alpha'}(t) \to \frac{k^2_1 - k^2_2}{2 k_1} K_{\pm}(t),
\end{equation*}
where
\begin{align*}
&K_{\pm}(t) = t \pm \sqrt{t^2 + \hat{C}},\\
&\hat{C} = 4 k_1 k_2\left( \frac{1}{x^2 (k_1-k_2)^2}-\frac{1}{(k_1+k_2)^2}\right)
\end{align*}
and the sign $\pm$ is that of the quantity $k^2_1 - k^2_2$. Thus, by the convergence of the K\"ahler forms $\omega_{\alpha'}$, the curvature forms $F_{\alpha'}$ also have a smooth limit $F$ as $\alpha' \to \infty$. 

Finally we may write down the equations satisfied by the limit K\"ahler form $\omega$ and curvature form $F$. By our previous results we have expansions, as $\alpha' \to \infty$,
\begin{align*}
&\im\left(e^{-\ii\hth_{\alpha'}}( \omega_{\alpha'} - \ii \alpha' F_{\alpha'})^2\right) = \left(Z_1 \omega_{\alpha'}\wedge F_{\alpha'}-Z_2 F^2_{\alpha'}\right)\alpha' + O(1),\\
&\re\left(e^{-\ii\hth_{\alpha'}}( \omega_{\alpha'} - \ii \alpha' F_{\alpha'})^2\right) = F^2_{\alpha'} (\alpha')^2 + O(1),
\end{align*}
for some cohomological constants $Z_1, Z_2$. Similarly, 
\begin{align*}
&\alpha_{\alpha'} = \frac{|k^2_1 - k^2_2|}{2(k_1 - k_2)^2 k^2_2}\frac{\lp -2 + s_{\Sigma}x \rp}{(\alpha')^2} + O\left(\frac{1}{(\alpha')^3}\right),\\
&\hat{r}_{\alpha'} = |k^2_1 - k^2_2|(\alpha')^2 + O(\alpha').
\end{align*} 
Thus, passing to the limit as $\alpha' \to \infty$ in the equations \eqref{coupled_dHYM_Intro_scale}, we find that $\omega$, $F$ satisfy the equations
\begin{equation*}
\begin{cases}
F \wedge \omega = c_1 F^2\\
s(\omega) - \alpha_{\infty}\frac{F^2}{\omega^2}= c_2  
\end{cases}
\end{equation*}
for a unique $\alpha_{\infty}$ and cohomological constants $c_1$, $c_2$. Using the first equation, the second can also be written in the twisted cscK form as
\begin{equation*}
s(\omega) - \hat{\alpha} \Lambda_{\omega} F = c_2  
\end{equation*}
for some unique $\hat{\alpha}$. This completes the proof of Theorem \ref{SmallRadiusThm}.


\begin{thebibliography}{}
\bibitem{AGG} L. \'Alvarez-C\'onsul, M. Garcia-Fernandez and O. Garc\'ia-Prada, \emph{Coupled equations for K\"ahler metrics and Yang-Mills connections}, Geometry and Topology 17 (2013) 2731--2812.
\bibitem{branes} P. Aspinwall et al., \emph{Dirichlet branes and mirror symmetry}. Clay Mathematics Monographs, 4. American Mathematical Society, Providence, RI; Clay Mathematics Institute, Cambridge, MA, 2009. x+681 pp. ISBN: 978-0-8218-3848-8.
\bibitem{collinsJacobYau} T. Collins, A. Jacob and S.-T. Yau, \emph{$(1,1)$ forms with specified Lagrangian phase: a priori estimates and algebraic obstructions}, Camb. J. Math. 8 (2020), no. 2, 407--452. 
\bibitem{collinsShi} T. Collins and Y. Shi, \emph{Stability and the deformed Hermitian-Yang-Mills equation}, arXiv:2004.04831 [math.DG].
\bibitem{collinsXie} T. Collins, D. Xie and S.-T. Yau, \emph{The deformed Hermitian-Yang-Mills equation in geometry and physics}, 	Geometry and Physics, Volume I:  A Festschrift in honour of Nigel Hitch, Oxford University Press, 2018.  
\bibitem{collinsYau} T. Collins and S.-T. Yau, \emph{Moment maps, nonlinear PDE, and stability in mirror symmetry}, 	arXiv:1811.04824 [math.DG].
\bibitem{datar} V. Datar and V. Pingali, \emph{On coupled constant scalar curvature K\"ahler metrics}, J. Symplect. Geom., vol. 18, no. 4, 961--994 (2020).  
\bibitem{donaldson} S. K. Donaldson, \emph{Remarks on gauge theory, complex geometry and 4-manifold topology}, in Fields Medallists' lectures, volume 5 of World Sci. Ser. 20th Century Math., pages 384--403. World Sci. Publ., River Edge, NJ, 1997.
\bibitem{feix} B. Feix, \emph{Hyperk\"ahler metrics on cotangent bundles}, J. Reine Angew. Math., 532:33--46, 2001.
\bibitem{fujiki} A. Fujiki, \emph{Moduli space of polarized algebraic manifolds and K\"ahler metrics}, Sugaku Expositions, 5(2):173--191, 1992.
\bibitem{hanJinStab} X. Han and X. Jin, \emph{Stability of line bundle mean curvature flow}, arXiv:2001.07406 [math.DG].
\bibitem{hanJinChern} X. Han and X. Jin, \emph{Chern number inequalities of deformed Hermitian-Yang-Mills metrics on four dimensional K\"ahler manifolds}, arXiv:2008.06862 [math.DG].
\bibitem{wytt} J. Hultgren and D. Witt Nystr\"om, \emph{Coupled K\"ahler-Einstein metrics}, International Mathematics Research Notices, rnx298, https://doi.org/10.1093/imrn/rnx298.
\bibitem{hwang} A. Hwang and M. Singer, \emph{A momentum construction for circle-invariant K\"ahler metrics}, Trans. Amer. Math. Soc. 354 (2002), 2285--2325. 
\bibitem{jacob} A. Jacob and S.-T. Yau, \emph{A special Lagrangian type equation for holomorphic line bundles}, Math. Ann. 369 (2017), no. 1-2, 869--898.
\bibitem{keller} J. Keller and C. T\"onnesen-Friedman, \emph{A non trivial example of coupled equations for K\"ahler metrics and Yang-Mills connections}, Central European Journal of Math., 10(5), 1673--1687 (2012).
\bibitem{leungYau} N. C. Leung, S.-T. Yau and E. Zaslow, \emph{From special Lagrangian to Hermitian-Yang-Mills via Fourier-Mukai}, Adv. Theor. Math. Phys. 4 (2000), no. 6, 1319--1341.
\bibitem{marino} M. Mari\~no, R. Minasian, G. Moore, and A. Strominger, \emph{Nonlinear instantons from supersymmetric p-branes}, J. High Energy Phys. (2000), no. 1.
\bibitem{us} E. Schlitzer and J. Stoppa, \emph{Deformed Hermitian Yang-Mills connections, extended gauge group and scalar curvature}, arXiv:1911.10852 [math.DG].  
\bibitem{takaCollapse} R. Takahashi, \emph{Collapsing of the line bundle mean curvature flow on K\"ahler surfaces}, 	arXiv:1912.13145 [math.DG].
\bibitem{takaTan} R. Takahashi, \emph{Tan-concavity property for Lagrangian phase operators and applications to the tangent Lagrangian phase flow}, arXiv:2002.05132 [math.DG].
\bibitem{richard} R. P. Thomas, \emph{Moment maps, monodromy, and mirror manifolds}, Symplectic geometry and mirror symmetry (Seoul, 2000), 467--498, World Sci. Publ., River Edge, NJ, 2001.
\bibitem{richardSurvey} R. P. Thomas, \emph{The geometry of mirror symmetry}, in Encyclopaedia of Mathematical Physics, eds. J.-P. Fran\c{c}oise, G.L. Naber and Tsou S.T. Oxford: Elsevier, 2006.
\end{thebibliography}
\end{document}